\documentclass[a4paper,reqno]{siamltex}

\usepackage{amssymb}
\usepackage{amsmath}
\usepackage{graphicx}
\usepackage[numbers]{natbib}
\usepackage{stmaryrd}

\usepackage{verbatim}

\newcommand{\vect}[1]{\boldsymbol{#1}}

\newcommand{\dif}{\, d}

\newcommand{\pd}{\partial}

\newcommand{\brac}[1]{\left( {#1} \right)}
\newcommand{\bracc}[1]{\left\{ {#1} \right\}}

\newcommand{\jump}[1]{\llbracket {#1} \rrbracket}
\newcommand{\avg}[1]{\langle {#1} \rangle}

\newcommand{\norm}[1]{\left\| {#1} \right\|}
\newcommand{\snorm}[1]{\left| {#1} \right|}

\newcommand{\unorm}[1]{%
  \left\vert\kern-0.9pt\left\vert\kern-0.9pt\left\vert #1
    \right\vert\kern-0.9pt\right\vert\kern-0.9pt\right\vert}

\title{Analysis of an interface stabilised finite element method: The
       advection-diffusion-reaction equation}
\author{Garth N. Wells\thanks{Department of Engineering,
University of Cambridge, Trumpington Street, Cambridge CB2 1PZ, United Kingdom
 ({\tt gnw20@cam.ac.uk})}.}

\date{\today}

\begin{document}
\maketitle
\begin{abstract}
Analysis of an interface stabilised finite element method for the scalar
advection-diffusion-reaction equation is presented.
The method inherits attractive properties of both continuous
and discontinuous Galerkin methods, namely the same number of global
degrees of freedom as a continuous Galerkin method on a given mesh and the
stability properties of discontinuous
Galerkin methods for advection
dominated problems. Simulations using the approach in
other works demonstrated good stability properties with minimal
numerical dissipation, and standard convergence rates for the lowest order
elements were observed. In this work,
stability of the formulation, in the form of an inf-sup
condition for the hyperbolic limit and coercivity for the elliptic case,
is proved,
as is order $k +1/2$ order convergence for the advection-dominated case
and order $k +1$ convergence for the diffusive limit in the $L^{2}$ norm.
The analysis results are supported by a number of numerical
experiments.
\end{abstract}


\pagestyle{myheadings}
\thispagestyle{plain}
\markboth{GARTH N. WELLS }{AN INTERFACE STABILISED FINITE ELEMENT METHOD}
\section{Introduction}
\label{sec:introduction}
Discontinuous Galerkin methods have proven effective and popular for classes of
partial differential equations, in particular transport equations
in which advection is dominant.
The attractive stability properties of suitably constructed discontinuous
Galerkin methods and the possibility of matching non-conforming meshes are advantageous,
but do come at the cost of an increased number of global degrees of freedom
on a given mesh compared to continuous Galerkin methods. In a number of recent
works, advances have been made in reconciling the appealing features of
continuous and discontinuous Galerkin methods in one framework.
Works in this direction include those of \citet{hughes:2006},
\citet{labeur:2007} and \citet{cockburn:2009} for the advection-diffusion
equation, \citet{burman:2007} for advection-reaction equation,
and \citet{labeur:2007} and \citet{labeur:2009}
for the incompressible Navier-Stokes equations.
These methods generally strive for a reduction in the
number of global degrees of freedom relative to a conventional
discontinuous Galerkin method without sacrificing other desirable
features. In this work, stability and convergence estimates are presented
for one such method applied to the scalar
advection-diffusion-reaction equation, namely the interface stabilised method
as formulated in \citet{labeur:2007}.

The principle behind the interface stabilised method is
simple: the equation of interest is posed cell-wise subject to weakly
imposed Dirichlet boundary conditions in the spirit of discontinuous Galerkin
methods. The boundary condition which is weakly satisfied is provided
by an `interface' function that lives only on cell facets and is
single-valued on cell facets.
An equation for this additional field is furnished by
insisting upon weak continuity of the so-called `numerical flux'  across cell facets.
This weak continuity of the numerical flux is in contrast with typical
discontinuous Galerkin methods which satisfy continuity
of the numerical flux across cell facets point-wise by construction.
For particular choices in the method,
it may be possible to achieve point-wise continuity.
Upwinding of the advective flux at interfaces can be incorporated
naturally in the definition of the numerical flux,
as is typical for discontinuous Galerkin methods.
By building a degree of continuity into the interface function spaces (at
cell vertices in two dimensions and across cell edges in three dimensions), the number of global
degrees of freedom is equal to that for a continuous Galerkin method on the same mesh.
The key to this reduction in the number of global degrees of freedom
is that functions which are defined on cells are not linked directly
across cell facets, rather they communicate only
via the interface function. Therefore, functions on cells can be eliminated
locally (cell-wise) in favour of the functions that live on cell facets.
Outwardly the approach appears to have elements in common with mortar
methods, and could serve to elucidate links between mortar and
discontinuous Galerkin methods.

The motivation for analysing the interface stabilised method
comes from the observed performance of the method for the advection-diffusion
in \citet{hughes:2006} and \citet{labeur:2007} and for the
incompressible Navier-Stokes equations in \citet{labeur:2007},
and for the Navier-Stokes equations on moving domains,
as presented in \citet{labeur:2009}. The method was observed in simulations
to be robust and only  minimal numerical dissipation could be detected.
\citet{labeur:2007} also showed that the methodology can lead to a stable formulation
for Stokes equation using equal-order Lagrange basis functions for the velocity
and the pressure.
The method examined in this work is closely related to that
formulated by \citet{hughes:2006} for the advection-diffusion equation,
and analysed in \citet{buffa:2006}.
\citet{buffa:2006}
proved stability for a streamline-diffusion
stabilised variant of the method, but not for the original
formulation. For the case without the additional
streamline diffusion term,  stability was demonstrated
for some computed examples by evaluating the inf-sup condition numerically.
However, in the absence of an analytical stability estimate convergence estimates
could not be formulated.
The stability and error estimates developed here for a method without an additional
streamline diffusion term
are made possible
by: (1) the different and transparent format in which the problem is posed; and (2)
the different machinery that is brought to bear on the problem. With respect to the
last point, advantage is taken of some
developments formulated by~\citet{ern:2006}.

In the remainder of this work, the equation of interest and the
numerical method to be analysed are first formalised. This is followed by analysis
of the hyperbolic case, for which satisfaction of an inf-sup is demonstrated. The
the diffusive limit case is then considered, for which demonstration
of coercivity suffices. The results of some numerical simulations are then
presented in support of the analysis, after which conclusions are drawn.

\section{Interface stabilised method}
\label{sec:preliminaries}
%
\subsection{Model problem}
Consider a polygonal domain $\Omega \subset \mathbb{R}^{d}$,
where $1 \le d \le 3$, with boundary $\Gamma = \pd \Omega$.
The unit outward normal vector to the domain is denoted by~$\vect{n}$.
The advection-diffusion-reaction equation reads:
\begin{equation}
  \mu u + \vect{a} \cdot \nabla u - \kappa \nabla^{2} u = f  \quad {\rm in} \  \Omega,
\label{eqn:strong}
\end{equation}
where $\mu \ge 0$ and $\kappa \ge 0$ are assumed to be constant,
$\vect{a} : \Omega \rightarrow \mathbb{R}^{d}$ is a divergence-free vector
field that is Lipschitz continuous on $\Bar{\Omega}$
and satisfies $\norm{\vect{a}}_{L^{\infty}(\Omega)} \le 1$,
and $f : \Omega \rightarrow \mathbb{R}$
is a suitably regular source term.
The divergence-free condition on $\vect{a}$ can easily be relaxed
to $\mu - (1/2) \nabla \cdot \vect{a} > 0$.
Portions of the boundary on which $\vect{a} \cdot \vect{n} \ge 0$ are denoted by
$\Gamma_{+}$, and
portions on which $\vect{a} \cdot \vect{n} < 0$ are denoted by~$\Gamma_{-}$.
A function $\zeta$ is defined on boundaries such that
$\zeta = 0$ on outflow portions of the boundary ($\Gamma_{+}$) and
$\zeta =1$ on inflow portions of the boundary~($\Gamma_{-}$).

For the case $\kappa > 0$, the boundary is partitioned
into $\Gamma_{N}$ and $\Gamma_{D}$
such that $\overline{\Gamma_{N} \cup \Gamma_{D}} = \Gamma$ and
$\Gamma_{N} \cap \Gamma_{D} = \emptyset$, and
the boundary conditions
\begin{equation}
\begin{split}
  \brac{- \zeta u \vect{a} + \kappa \nabla u} \cdot \vect{n}
        &= g \quad {\rm on} \ \Gamma_{N},
\\
  u &= 0 \quad {\rm on} \ \Gamma_{D},
\end{split}
\label{eqn:strong_bcs}
\end{equation}
are considered, where $g : \Gamma_{N} \rightarrow \mathbb{R}$ is a
suitably smooth prescribed function.
For the case $\kappa = 0$, then $\Gamma_{D} = \emptyset$,
$\Gamma_{N} = \Gamma_{-}$ and
the considered boundary condition reads:
\begin{equation}
  -u \vect{a} \cdot \vect{n} = g \quad {\rm on} \ \Gamma_{-}.
\label{eqn:strong_bcs_hyperbolic}
\end{equation}
%
\subsection{The method}
Let $\mathcal{T}$ be a triangulation of $\Omega$
into non-overlapping simplices such that $\mathcal{T} = \bracc{K}$.
A simplex $K \in \mathcal{T}$ will be referred to as a cell and
a measure of the size of a cell $K$ will be denoted by~$h_{K}$,
with the usual assumption that $h_{K} \le 1$,
and $h = \max_{K\in \mathcal{T}} h_{K}$.
The boundary of a cell $K$ is denoted by $\pd K$ and
the outward unit normal to a cell is denoted by~$\vect{n}$.
The outflow portion of a cell boundary is the portion on
which $\vect{a} \cdot \vect{n} \ge 0$, and is denoted by~$\pd K_{+}$.
The inflow portion of a cell boundary is the portion on
which $\vect{a} \cdot \vect{n} < 0$, and is denoted by~$\pd K_{-}$.
As for the exterior boundary, the function $\zeta$
is defined such that $\zeta = 0$ on $\pd K_{+}$
and $\zeta =1$ on~$\pd K_{-}$.
The set of all facets $\mathcal{F} = \bracc{F}$ contained
in the mesh will be used, as will the union of all facets,
which is denoted by~$\Gamma^{0}$.
Adjacent cells are considered to share a common facet~$F$.

The bilinear and linear forms for the
advection-diffusion-reaction equation are now introduced.
Using the notation
$\vect{w} = \brac{w, \Bar{w}}$ and $\vect{v} = \brac{v, \Bar{v}}$,
consider the bilinear form:
\begin{multline}
  B\brac{\vect{w}, \vect{v}}
    =
     \int_{\Omega} \mu w v \dif x
    + \int_{\Omega} \brac{-\vect{a} w + \kappa \nabla w} \cdot \nabla v \dif x
\\
+ \sum_{K} \int_{\pd K} \brac{-\vect{a} w + \kappa \nabla w -\brac{\zeta \vect{a}
           - \frac{\alpha\kappa}{h_{K}}\vect{n}} \brac{\Bar{w} - w }}
          \cdot \vect{n} \brac{\Bar{v} - v} \dif s
\\
+ \sum_{K} \int_{\pd K} \kappa \brac{\Bar{w} - w} \nabla v \cdot \vect{n} \dif s
+ \int_{\Gamma_{+}} \vect{a} \cdot \vect{n} \Bar{v} \Bar{w} \dif s
\end{multline}
and the linear form
\begin{equation}
  L\brac{\vect{v}}
    = \int_{\Omega} f v \dif x + \int_{\Gamma_{N}} g \Bar{v} \dif s,
\end{equation}
where $\alpha \ge 0$. The relevant finite element function spaces for the problem
which will be considered read
\begin{align}
  W_{h}       &= \bracc{w_{h} \in L^{2}\brac{\Omega},
                        w_{h}|_{K} \in P_{k}\brac{K}
                        \forall K \in \mathcal{T}},
\\
  \Bar{W}_{h} &= \bracc{\Bar{w}_{h} \in H^{l}\brac{\Gamma^{0}},
                        \bar{w}_{h}|_{F} \in P_{k}\brac{F} \forall F \in \mathcal{F},
                        \bar{w}_{h} = 0 \ {\rm on} \ \Gamma_{D}},
\label{eqn:interface_space}
\end{align}
where $0 \le l \le 1$ and $P_{k}(K)$ denotes the space of standard Lagrange
polynomial functions of order $k$ on cell~$K$.
The space $W_{h}$ is the usual
space commonly associated with discontinuous Galerkin methods,
and the space $\Bar{W}_{h}$ contains Lagrange polynomial shape functions
that `live' only on cell facets and are single-valued on facets.
The choice of $l$, which determines the regularity of the facet
functions at cell vertices in two dimensions and across cell edges
in three dimensions, will have a significant impact
on the structure of the resulting matrix problem.
Using the notation $W^{\star}_{h} = W_{h} \times \Bar{W}_{h}$
and $\vect{v}_{h} = \brac{v_{h}, \Bar{v}_{h}}$,
the finite element problem of interest reads:
find $\vect{u}_{h} \in W_{h}^{\star}$ such that
\begin{equation}
  B\brac{\vect{u}_{h}, \vect{v}_{h}} = L\brac{\vect{v}_{h}}
      \quad \forall \vect{v}_{h} \in W^{\star}_{h}.
\label{eqn:fe_problem}
\end{equation}

To motivate the terms appearing in the bilinear form,
it is useful to consider the case in which
$v = 0$ and the case in which
$\Bar{v} = 0$ separately. Considering first $\vect{v}_{h} = \brac{v_{h}, 0}$,
the variational problem corresponding to
equation~\eqref{eqn:fe_problem} for a single cell reads:
given $\Bar{u}_{h} \in \Bar{W}_{h}$, find
$u_{h} \in W_{h}$ such that for all $v_{h} \in W_{h}$
\begin{multline}
    \int_{K} \mu u_{h} v_{h} \dif x
    + \int_{K}  \vect{\sigma}\brac{u_{h}} \cdot \nabla v_{h} \dif x
- \int_{\pd K} \Bar{\vect{\sigma}}\brac{\vect{u}_{h}} \cdot \vect{n}  v_{h} \dif s
\\
+ \int_{\pd K} \kappa \brac{\Bar{u}_{h} - u_{h}} \nabla v_{h} \cdot \vect{n} \dif s
    = \int_{K} f v_{h} \dif x,
\label{eqn:local_eqn}
\end{multline}
where $\vect{\sigma}\brac{w} = - \vect{a} \nabla w + \kappa \nabla w$
is the usual flux vector and $\Bar{\vect{\sigma}}\brac{\vect{w}}$ is
a `numerical flux',
\begin{equation}
  \Bar{\vect{\sigma}}\brac{\vect{w}} = -\vect{a} w + \kappa \nabla w -\brac{\zeta \vect{a}
           - \frac{\alpha\kappa}{h_{K}}\vect{n}} \brac{\Bar{w} - w }.
\label{eqn:numerical_flux}
\end{equation}
The problem in equation~\eqref{eqn:local_eqn} is essentially
a cell-wise postulation of a Galerkin problem for equation~\eqref{eqn:strong}
subject to the weak satisfaction of the boundary condition $u_{h} = \Bar{u}_{h}$.
In the numerical flux, the presence of the term $\zeta$ provides for
upwinding of the advective part of the flux, and the term
$\brac{\alpha\kappa/{h_{K}}}\vect{n} \brac{\Bar{w} - w}$ is
an interior penalty-type contribution to the numerical
flux~\citep{arnold:2002}.
The term $\int_{\pd K} \kappa \brac{\Bar{u}_{h} - u_{h}} \nabla v_{h} \cdot \vect{n} \dif s$
is typical of discontinuous Galerkin methods for elliptic problems,
and resembles that in~\citet{arnold:2002} for the Poisson equation.
The numerical flux can be evaluated on both sides of a facet.
On the outflow (upwind) portion of a cell boundary, the advective part
of the numerical flux is equal to the regular advective flux.
On the inflow (downwind) portion of a cell boundary, the
advective part of the numerical flux depends on the
interface function, taking on~$-\vect{a} \Bar{u}$.
The diffusive numerical flux on a cell boundary
has contributions from the
regular flux and a penalty-like contribution which depends on the difference
between $w_{h}$ and the interface function~$\Bar{w}_{h}$.
Setting  $v_{h} = 1$ in
equation~\eqref{eqn:local_eqn},
\begin{equation}
    \int_{K} \mu u_{h} \dif x
- \int_{\pd K} \Bar{\vect{\sigma}}\brac{\vect{u}_{h}} \cdot \vect{n} \dif s
    = \int_{K} f \dif x,
\end{equation}
which demonstrates local conservation in terms of the numerical flux.
Note that the numerical flux defined in equation~\eqref{eqn:numerical_flux} is not
single-valued on cell facets.
Setting $\vect{v}_{h} = \brac{0, \Bar{v}_{h}}$ furnishes the
problem: given $u_{h} \in W_{h}$, find $\Bar{u}_{h}\in \Bar{W}_{h}$ such that
for all $\Bar{v}_{h} \in \Bar{W}_{h}$
\begin{equation}
 \sum_{K} \int_{\pd K} \Bar{\vect{\sigma}}\brac{\vect{u}_{h}} \cdot \vect{n} \Bar{v}_{h} \dif s
+ \int_{\Gamma_{+}} \vect{a} \cdot \vect{n}  \Bar{u}_{h} \Bar{v}_{h} \dif  s
    =  \int_{\Gamma_{N}} g \Bar{v}_{h} \dif s,
\label{eqn:global_eqn}
\end{equation}
which is a statement of weak continuity of the numerical flux across cell facets.

Noteworthy in the bilinear form is that the functions $w_{h}$,
which are discontinuous across cell facets,
are not linked directly across facets. They are only linked
implicitly through their interaction with~$\Bar{w}_{h}$.
Setting $\vect{v}_{h} = \brac{v_{h}, 0}$ leads to a local (cell-wise) problem,
which, given $\Bar{u}_{h}$ and $f$ can be solved locally
to eliminate~$u_{h}$ in favour of~$\Bar{u}_{h}$.
This process is commonly referred to as static condensation.
Then, setting $\vect{v}_{h} = \brac{0, \Bar{v}_{h}}$, one
can solve a global problem to yield the interface solution $\Bar{u}_{h}$.
The field $u_{h}$ can then be recovered trivially element-wise.
To formulate a global problem with the same number of degrees as a continuous
finite element method, $l$ in equation~\eqref{eqn:interface_space} must
be chosen such that there is only one degree of freedom at a given point;
the interface functions are continuous at cell vertices
in two dimensions and along cell edges in three dimensions.
Further details on the formulation of the interface stabilised
method and various algorithmic details
can be found in \citet{labeur:2007}.

The formulation of \citet{hughes:2006} can be manipulated into framework presented in
this section, and in the hyperbolic limit coincides with the formulation presented
here. In the case of diffusion, \citet{hughes:2006} adopted an upwinded diffusive
flux whereas the diffusive flux is centred in the present method.
The formulation presented in
\citet{cockburn:2009} follows the same framework as \citet{labeur:2007}, although
the use of functions lying in $L^{2}\brac{\Gamma^{0}}$ on facets is
advocated.

The method is now shown to be consistent with equation~\eqref{eqn:strong}.
If $u$ solves equation~\eqref{eqn:strong}, it is chosen to define
$\vect{u} = \brac{u, u}$. The action of the trace operator
in the second slot is implicit in this
definition (this will be expanded upon in Section~\ref{sec:notation}).
With this definition of $\vect{u}$ consistency can be addressed.
\begin{lemma}[consistency]\label{lemma:consistency}
If $\vect{u} = \brac{u, u}$, where $u \in H^{m}\brac{\Omega}$
is a solution to~\eqref{eqn:strong} with
$m=2$ if $\kappa > 0$ and $m=1$ otherwise,
and if $\vect{u}_{h}$ solves~\eqref{eqn:fe_problem},
then for all $\vect{v}_{h} \in W_{h}^{\star}$
\begin{equation}
  B\brac{\vect{u} - \vect{u}_{h}, \vect{v}_{h}} = 0.
\end{equation}
\end{lemma}
\begin{proof}
Since $\vect{u}_{h}$ is a solution to~\eqref{eqn:fe_problem} and due
to the bilinear nature of $B$, it suffices to demonstrate
that $B\brac{\vect{u}, \vect{v}_{h}} - L\brac{\vect{v}_{h}} = 0$.
Considering first $B\brac{\vect{u}, \brac{v_{h}, 0}} - L\brac{\brac{v_{h}, 0}}$,
which is presented in
equation~\eqref{eqn:local_eqn}, after applying integration by parts
\begin{equation}
B\brac{\vect{u}, \brac{v_{h}, 0}} - L\brac{\brac{v_{h}, 0}}
   = \int_{K} \brac{\mu u + \vect{a} \nabla u - \kappa \nabla^{2} u - f}v_{h} \dif x = 0,
\label{eqn:local_consistency}
\end{equation}
since $u$ satisfies~\eqref{eqn:strong} for $\kappa = 0$.
Considering now
$B\brac{\vect{u}, \brac{0, \Bar{v}_{h}}} - L\brac{\brac{0, \Bar{v}_{h}}}$, which is presented in
equation~\eqref{eqn:global_eqn},
\begin{equation}
B\brac{\vect{u}, \brac{0, \Bar{v}_{h}}} - L\brac{\brac{0, \Bar{v}_{h}}}
   = \int_{\Gamma_{N}} \brac{\brac{-\zeta u \vect{a} + \kappa \nabla u} \cdot \vect{n}  - g} \Bar{v}_{h} \dif  s = 0,
\label{eqn:global_consistency}
\end{equation}
since $u$ satisfies the boundary condition in~\eqref{eqn:strong_bcs_hyperbolic}.
Summing equations~\eqref{eqn:local_consistency}
and~\eqref{eqn:local_consistency}
and subtracting $B\brac{\vect{u}_{h}, \vect{v}_{h}} - L\brac{\vect{v}_{h}}=0$
concludes the proof.
\end{proof}
\subsection{Limit cases}
The method will be analysed for the hyperbolic ($\kappa = 0$)
and elliptic ($\vect{a} = \vect{0}, \mu =0$) limit cases.
The bilinear form is therefore decomposed into advective and diffusive parts,
\begin{equation}
  B\brac{\vect{w}, \vect{v}}
      = B_{A}\brac{\vect{w}, \vect{v}} + B_{D}\brac{\vect{w}, \vect{v}},
\end{equation}
where
\begin{multline}
  B_{A}\brac{\vect{w}, \vect{v}}
    =
    \int_{\Omega} \mu w v \dif x
    - \sum_{K} \int_{K} \vect{a} w \cdot  \nabla v \dif x
 - \sum_{K} \int_{\pd K_{+}} \vect{a}\cdot \vect{n} w   \brac{\Bar{v} - v} \dif s
\\
 - \sum_{K} \int_{\pd K_{-}} \vect{a}\cdot \vect{n} \Bar{w} \brac{\Bar{v} - v} \dif s
+ \int_{\Gamma_{+}} \vect{a} \cdot \vect{n} \Bar{w} \Bar{v} \dif s
\label{eqn:B_hyperbolic}
\end{multline}
and
\begin{multline}
  B_{D}\brac{\vect{w}, \vect{v}}
    = \sum_{K} \int_{K} \kappa \nabla w \cdot \nabla v \dif x
\\
   + \sum_{K} \int_{\pd K} \brac{\kappa \nabla w
        + \frac{\alpha\kappa}{h_{K}}\vect{n} \brac{\Bar{w} - w }}
          \cdot \vect{n} \brac{\Bar{v} - v} \dif s
\\
+ \sum_{K} \int_{\pd K} \kappa \brac{\Bar{w} - w} \nabla v \cdot \vect{n} \dif s.
\label{eqn:B_diffusive}
\end{multline}
Stability and error estimates will be proved by analysing
$B_{A}\brac{\vect{w}, \vect{v}}$ and $B_{D}\brac{\vect{w}, \vect{v}}$
independently.
\subsection{Conventional discontinuous Galerkin methods as a special case}
If the functions defined on facets are defined to be in
$L^{2}\brac{\Gamma^{0}}$ ($l=0$ in equation~\eqref{eqn:interface_space}),
then for the hyperbolic case the formulation reduces to the conventional
discontinuous Galerkin formulation
with full upwinding of the advective flux~\citep{reed:1973,lesaint:1974,johnson:1986}.
In the diffusive limit,
it reduces to a method which closely resembles the symmetric interior
penalty method \citep{wheeler:1978,arnold:1982}.
Of prime practical interest is the case where the
interface functions are continuous as this
leads to the fewest number of global degrees of freedom, but the special case
of $l=0$
is considered briefly in this section to illustrate a link with conventional
discontinuous Galerkin methods.

For the case $\mu = \kappa = 0$, setting $v_{h} = 0$ everywhere and
$\Bar{v}_{h} = 0$ everywhere with the exception of one interior facet
$F$, the method implies that at the facet $F$
\begin{equation}
   \int_{F} \vect{a} \Bar{w}_{h} \cdot \vect{n}_{+} \Bar{v}_{h} \dif s
        = \int_{F_{-}} \vect{a} w_{h+} \cdot \vect{n}_{+} \Bar{v}_{h} \dif s,
\end{equation}
where the subscript `$+$' indicates functions evaluated on the
boundary of the upwind cell.
This implies that for a given $w_{h}$, the facet function
$\Bar{w}_{h}$ simply takes on the upwind value on each facet.
Inserting this into equation~\eqref{eqn:B_hyperbolic} and
setting~$\Bar{v}_{h} = 0$,
\begin{multline}
  B_{A}\brac{w_{h}, v_{h}}
    =
     \int_{\Omega} \mu w_{h} v_{h} \dif x
    - \sum_{K} \int_{K} \vect{a} w_{h} \cdot  \nabla v_{h} \dif x
\\
 + \sum_{K} \int_{\pd K_{+}} \vect{a}\cdot \vect{n} w_{h} v_{h} \dif s
 + \sum_{K} \int_{\pd K_{-}} \vect{a} \cdot \vect{n} w_{h+} v_{h} \dif s,
\end{multline}
which is the bilinear form associated with the
classical discontinuous Galerkin formulation
for hyperbolic problems with full upwinding.

The diffusive case ($\kappa =1$, $\mu = 0$, $\vect{a} = 0$,
$\alpha > 0$) is now considered, in which case the subscripts '$+$'
and '$-$' indicate functions evaluated on opposite sides of a
facet.
Following the same process as for the hyperbolic case leads to
\begin{multline}
  \int_{F} \frac{\alpha}{h_{K}} \Bar{w}_{h} \Bar{v}_{h} \dif s
 =
\frac{1}{2} \int_{F_{-}} \brac{-\nabla w_{h-}\cdot \vect{n}_{-}
     + \frac{\alpha}{h_{K}} w_{h-}} \Bar{v}_{h} \dif s
\\
+
\frac{1}{2} \int_{F_{+}} \brac{- \nabla w_{h+}\cdot \vect{n}_{+}
     + \frac{\alpha}{h_{K}} w_{h+}} \Bar{v}_{h} \dif s
\end{multline}
on facets. Assuming for simplicity that $h_{K}$ is constant,
inserting the expression for $\Bar{w}_{h}$ into \eqref{eqn:B_diffusive} and
after some tedious manipulations, the bilinear forms reduces to:
\begin{multline}
  B_{D}\brac{w_{h}, v_{h}}
    = \sum_{K} \int_{K} \nabla w_{h} \cdot \nabla v_{h} \dif x
- \int_{\Gamma^{0}} \avg{\nabla w_{h}} \cdot \jump{v_{h}} \dif s
\\
- \int_{\Gamma^{0}} \jump{w_{h}}\cdot \avg{\nabla v_{h}} \dif s
+ \frac{\alpha}{2h_{K}} \int_{\Gamma^{0}} \jump{w_{h}}\cdot \jump{v_{h}} \dif s
- \frac{h_{K}}{2\alpha} \int_{\Gamma^{0}} \jump{\nabla w_{h}}\jump{\nabla v_{h}} \dif s,
\end{multline}
where $\avg{a} = 1/2\brac{a_{+} + a_{-}}$
and $\jump{a} = \brac{a_{+}\vect{n}_{+} + a_{-}\vect{n}_{-}}$
are the usual average and jump definitions, respectively.
This bilinear form resembles closely that of the
conventional symmetric interior  penalty method,
with the exception of the term which penalises jumps in the gradient of the
solution.
\section{Notation and useful inequalities}
\label{sec:notation}
The standard norm on the Sobolev space $H^{s}(K)$ will be denoted by
$\norm{\cdot}_{s, K}$ and the $H^{s}(K)$ semi-norm  will be denoted by
$\snorm{\cdot}_{s, K}$.
Constants $c$ which are independent of $h_{K}$ will be used extensively in
the presentation.
The values of constants without subscripts may change at each appearance,
and the value of any constant with a numeral subscript remains fixed.
When $c$ appears with a parameter subscript, this indicates a dependence on
a model parameter. For example, $c_{\mu}$ indicates a dependence on~$\mu$.

Use will be made of various estimates for functions
on finite element cells for the case
$h_{K} \le 1$. In particular, use will be made of the
trace inequalities~\citep{arnold:1982,ern:book}
\begin{align}
  \norm{v}^{2}_{0,\pd K}
    &\le c\brac{h_{K}^{-1} \norm{v}^{2}_{0,K}
       + h_{K} \snorm{v}^{2}_{1, K}} \quad \forall v \in H^{1}(K),
\label{eqn:dg_trace_orig}
\\
  \norm{\nabla v \cdot \vect{n}}^{2}_{0,\pd K} &
      \le c\brac{h_{K}^{-1} \snorm{v}_{1,K}^{2}
       + h_{K} \snorm{v}_{2, K}^{2}} \quad \forall v \in H^{2}(K).
\label{eqn:dg_trace_normal}
\end{align}
On polynomial finite element spaces, the inverse
estimate~\citep{brenner:book,ern:book}
\begin{equation}
  \snorm{v_{h}}_{1,K} \le c h_{K}^{-1} \norm{v_{h}}_{0,K}
    \quad \forall v_{h} \in P_{k}(K)
\label{eqn:inverse_gradient}
\end{equation}
will be used extensively.
Combining equations~\eqref{eqn:dg_trace_orig}
and~\eqref{eqn:inverse_gradient} leads to
\begin{equation}
  \norm{v_{h}}_{0,\pd K} \le c h_{K}^{-\frac{1}{2}} \norm{v_{h}}_{0,K}  \quad \forall v_{h} \in P_{k}(K).
\label{eqn:dg_trace}
\end{equation}

Frequently, functions defined on $\Omega$ or on a finite element
cell $K$ will be restricted to an interior or exterior boundary.
For finite element functions defined on a cell, owing to the continuity
of the functions on a cell the trace is well-defined point-wise on the
cell boundary.
When considering functions
in $H^{s}\brac{\Omega}$ restricted to $\Gamma^{0}$, the action of  a
trace operator  $\gamma: H^{s}\brac{\Omega} \rightarrow H^{s-1/2}\brac{\Gamma^{0}}$
should be taken as implied in the presentation.
\section{Analysis for the hyperbolic limit}
The interface stabilised method is first analysed for the hyperbolic limit case
which corresponds to the bilinear form
in equation~\eqref{eqn:B_hyperbolic}. For this case the spaces
\begin{align}
  W(h)      &= W_{h} + H^{1}\brac{\Omega}, \\
 \Bar{W}(h) &= \bar{W}_{h} + H^{1/2}\brac{\Gamma^{0}},
\end{align}
will be used in the analysis, as will the notation
$W^{\star}(h) = W(h) \times \Bar{W}(h)$. The space
$\Bar{W}(h)$ has been defined such that it contains the
trace of all functions in~$H^{1}\brac{\Omega}$ on~$\Gamma^{0}$.
This will prove important in developing error estimates.

Introducing the notation $a_{n} = |\vect{a} \cdot \vect{n}|$,
two norms are defined on $W^{\star}(h)$.
The first is what will be referred to as the `stability' norm,
\begin{equation}
 \unorm{\vect{v}}_{A}^{2}
     = \mu \norm{v}^{2}_{0,\Omega}
           + \sum_{K} h_{K} \norm{\vect{a} \cdot \nabla v}^{2}_{0,K}
           + \sum_{K} \norm{a_{n}^{\frac{1}{2}} \brac{\Bar{v} - v} }^{2}_{0,\pd K}
           + \norm{a_{n}^{\frac{1}{2}} \Bar{v} }^{2}_{0, \Gamma}.
\label{eqn:stability_norm}
\end{equation}
The second norm, which will be referred to as the `continuity' norm,
reads
\begin{equation}
 \unorm{\vect{v}}^{2}_{A^{\prime}}
     = \unorm{\vect{v}}_{A}^{2}
           + \sum_{K} h_{K}^{-1} \norm{v}^{2}_{0, K}
           + \sum_{K} \norm{a_{n}^{\frac{1}{2}} \Bar{v}}^{2}_{0, \pd K_{-}}
           + \sum_{K} \norm{a_{n}^{\frac{1}{2}} v}^{2}_{0, \pd K_{+}}.
\label{eqn:error_norm}
\end{equation}
Control of $\vect{v}_{h} \in W^{\star}_{h}$ in terms of the $\unorm{\cdot}_{A}$ norm
also implies control of $h_{K}\norm{a_{n}^{\frac{1}{2}} \Bar{v}_{h}}_{0, \pd K}^{2}$
due to the following proposition.
\begin{proposition}
There exists a constant $c > 0$ such that for all $K \in \mathcal{T}$ and
for all $\vect{v}_{h} \in W^{\star}_{h}$
\begin{equation}
  h_{K} \norm{\Bar{v}_{h}}^{2}_{0, \pd K}
  \le  c \brac{\norm{\Bar{v}_{h} - v_{h}}^{2}_{0, \pd K} + \norm{v_{h}}^{2}_{0, K} }.
\end{equation}
\end{proposition}
\begin{proof}
Using the triangle inequality and the inverse inequality~\eqref{eqn:dg_trace}:
\begin{equation}
\begin{split}
  h_{K} \norm{\Bar{v_{h}}}^{2}_{0, \pd K}
    =& h_{K} \norm{\Bar{v}_{h} - v_{h} + v_{h}}^{2}_{0, \pd K}
\\
   \le& h_{K} \brac{\norm{\Bar{v}_{h} - v_{h}}_{0, \pd K} + \norm{v_{h}}_{0, \pd K}}^{2}
\\
  \le& 2h_{K} \brac{ \norm{\Bar{v}_{h} - v_{h}}^{2}_{0, \pd K}
         + ch_{K}^{-1} \norm{v_{h}}^{2}_{0, K} }
\\
  \le&  c \brac{\norm{\Bar{v}_{h} - v_{h}}^{2}_{0, \pd K} + \norm{v_{h}}^{2}_{0, K} }.
\end{split}
\end{equation}
\end{proof}
\subsection{Stability}
Stability of the interface stabilised method for hyperbolic problems will be
demonstrated through satisfaction of the  inf-sup condition. Before considering
the inf-sup stability, a number of intermediate results are presented. The analysis
borrows from the approach of \citet{ern:2006} to
discontinuous Galerkin methods (see also  \citet[Section 5.6]{ern:book}). A
similar approach is adopted by \citet{burman:2007}.
\begin{lemma}[coercivity]
For all $\vect{v} \in W^{\star}(h)$
\begin{equation}
  B_{A}\brac{\vect{v}, \vect{v}}
\ge
\mu \norm{v}^{2}_{0, \Omega}
           + \frac{1}{2}\sum_{K} \norm{a_{n}^{\frac{1}{2}} \brac{\Bar{v} - v} }^{2}_{0,\pd K}
           + \frac{1}{2}\norm{a_{n}^{\frac{1}{2}} \Bar{v}}^{2}_{0,\Gamma}.
\label{eqn:coercivity}
\end{equation}
\end{lemma}
\begin{proof}
From the definition of $B_{A}\brac{\vect{v}, \vect{v}}$ and
the fact that $\vect{a}$ is divergence-free, it follows from
the application of integration by parts to~\eqref{eqn:B_hyperbolic}
and
some straightforward manipulations that
\begin{equation}
  B_{A}\brac{\vect{v}, \vect{v}} = \mu \norm{v}^{2}_{0, \Omega}
           + \frac{1}{2}\sum_{K} \norm{a_{n}^{\frac{1}{2}} \brac{\Bar{v} - v} }^{2}_{0, \pd K}
           + \frac{1}{2}\norm{ a_{n}^{\frac{1}{2}} \Bar{v} }^{2}_{0, \Gamma}.
\label{eqn:B_A_vv}
\end{equation}
\end{proof}

As is usual for advection-reaction problems,
$B_{A}\brac{\vect{v}, \vect{v}}$ is coercive with respect to a particular
norm, but the norm
offers no control over derivatives of the solution.

Consider a function $\vect{z}_{h}$ which depends on
$\vect{w}_{h} \in W^{\star}_{h}$ according to
\begin{equation}
\vect{z}_{h} = \brac{z_{h}, 0}
= \brac{-h_{K} \Bar{\vect{a}}_{K} \cdot \nabla w_{h}, 0},
\label{eqn:z_function}
\end{equation}
where $\Bar{\vect{a}}_{K}$ is the average of $\vect{a}$ on cell~$K$.
Lipschitz continuity of
$\vect{a}$ implies the following bound on a cell~$K$~\citep{burman:2007,evans:book}:
%
\begin{equation}
  \norm{\vect{a} - \Bar{\vect{a}}_{K}}_{L^{\infty}\brac{K}}
      \le  c h_{K} \snorm{\vect{a}}_{W_{\infty}^{1}\brac{K}}.
\label{eqn:a_bar_bound}
\end{equation}
\begin{lemma} \label{lemma:w_norm_bound}
If the function $\vect{z}_{h}$ depends on $\vect{w}_{h}$
according to equation~\eqref{eqn:z_function}, then
for all $\vect{w}_{h} \in  W^{\star}_{h}$
there exists a $c_{1} >0$ such that
if $\vect{v}_{h} = c_{1}\vect{w}_{h} + \vect{z}_{h}$, then
\begin{equation}
  \frac{1}{2} \unorm{\vect{w}_{h}}^{2}_{A}   \le
B_{A}\brac{\vect{z}_{h}, \vect{w}_{h}}  + c_{1} B_{A}\brac{\vect{w}_{h}, \vect{w}_{h}}
= B_{A}\brac{\vect{v}_{h}, \vect{w}_{h}}.
\label{eqn:B_bound}
\end{equation}
\end{lemma}
\begin{proof}
Consider first two bounds on $\norm{z_{h}}_{K}$.
Using equation~\eqref{eqn:a_bar_bound} and
the inverse estimate~\eqref{eqn:inverse_gradient},
\begin{equation}
\begin{split}
  \norm{z_{h}}_{0, K}
      & = \norm{h_{K} \Bar{\vect{a}}_{K} \cdot \nabla w_{h}}_{0, K}
\\
      &\le \norm{h_{K} \vect{a} \cdot \nabla w_{h}}_{0, K}
          + \norm{h_{K}\brac{\vect{a} - \Bar{\vect{a}}_{K}} \cdot \nabla w_{h}}_{0, K}
\\
      &\le \norm{h_{K} \vect{a} \cdot \nabla w_{h}}_{0, K}
          + c h_{K} \snorm{\vect{a}}_{W_{\infty}^{1}\brac{K}} \norm{h_{K} \nabla w_{h}}_{0, K}
\\
      &\le \norm{h_{K} \vect{a} \cdot \nabla w_{h}}_{0, K}
          + c \snorm{\vect{a}}_{W_{\infty}^{1}\brac{K}} \norm{h_{K} w_{h}}_{0, K},
\end{split}
\label{eqn:w_bound1}
\end{equation}
and from the inverse estimate~\eqref{eqn:inverse_gradient}
\begin{equation}
  \norm{z_{h}}_{0, K}  =  \norm{h_{K} \Bar{\vect{a}}_{K} \cdot \nabla w_{h}}_{0, K}
                    \le c \norm{\vect{a}}_{L^{\infty}(K)} \norm{w_{h}}_{0, K}
                    \le c \norm{w_{h}}_{0, K}.
\label{eqn:w_bound2}
\end{equation}

From the definition of the bilinear form in equation~\eqref{eqn:B_hyperbolic},
\begin{multline}
  \sum_{K} h_{K} \norm{\vect{a} \cdot \nabla w_{h}}^{2}_{0, K}
  = B_{A}\brac{\vect{z}_{h}, \vect{w}_{h}}
    + \sum_{K} h_{K} \int_{K} \mu \Bar{\vect{a}}_{K} \cdot \brac{\nabla w_{h}} w_{h} \dif x
\\
    + \sum_{K} h_{K}  \int_{K}\brac{\vect{a} \cdot \nabla w_{h}} \brac{\vect{a}
                   - \Bar{\vect{a}}}_{K} \cdot \nabla w_{h} \dif x
\\
       -  \sum_{K}h_{K}\int_{\pd K_{+}} a_{n}  \brac{\Bar{w}_{h} - w_{h}}
                       \Bar{\vect{a}}_{K}\cdot\nabla w_{h} \dif s.
\end{multline}
Applying the Cauchy-Schwarz inequality to the various terms on the right-hand side,
\begin{multline}
  \sum_{K} h_{K} \norm{\vect{a} \cdot \nabla w_{h}}^{2}_{0, K}
  \le B_{A}\brac{\vect{z}_{h}, \vect{w}_{h}}
    + \sum_{K} \norm{\mu w_{h} }_{0, K} \norm{h_{K} \Bar{\vect{a}}_{K} \cdot \nabla w_{h}}_{0, K}
\\
    + \sum_{K} \norm{h_{K}^{\frac{1}{2}}\vect{a}\cdot\nabla w_{h}}_{0, K}
      \norm{h_{K}^{\frac{1}{2}}\brac{\vect{a} - \Bar{\vect{a}}}_{K} \cdot \nabla w_{h}}_{0, K}
\\
   + \sum_{K} \norm{h_{K} \Bar{\vect{a}}_{K}\cdot\nabla w_{h}}_{0, \pd K_{+}}
             \norm{a_{n}\brac{\Bar{w}_{h} - w_{h}}}_{0, \pd K_{+}}.
\end{multline}
Each term is now appropriately bounded.
Using equation~\eqref{eqn:w_bound2},
\begin{equation}
  \sum_{K} \norm{\mu w_{h}}_{0, K} \norm{h_{K} \Bar{\vect{a}}_{K} \cdot \nabla w_{h}}_{0, K}
   \le c \mu \sum_{K} \norm{w_{h}}^{2}_{0, K}.
\end{equation}
Setting $R_{2} = \norm{h_{K}^{\frac{1}{2}}\vect{a}\cdot\nabla w_{h}}_{0, K}
      \norm{h_{K}^{\frac{1}{2}}\brac{\vect{a} - \Bar{\vect{a}}_{K}}\cdot \nabla w_{h}}_{0, K}$
and using \eqref{eqn:a_bar_bound}, an inverse inequality and Young's
inequality,
\begin{equation}
\begin{split}
    R_{2}
 \le& c \norm{h_{K}^{\frac{1}{2}}\vect{a}\cdot\nabla w_{h}}_{0, K}
             \norm{\vect{a} - \Bar{\vect{a}}_{K}}_{L^{\infty}(K)}
      \norm{h_{K}^{\frac{1}{2}}\nabla w_{h}}_{0, K}
\\
  \le& c h_{K} \snorm{\vect{a}}_{W^{1}_{\infty}(K)} \norm{h_{K}^{\frac{1}{2}}\vect{a}\cdot\nabla w_{h}}_{0, K}
                     \norm{h_{K}^{\frac{1}{2}}\nabla w_{h}}_{0, K}
\\
   \le& c h_{K} \snorm{\vect{a}}_{W^{1}_{\infty}(K)}\brac{\frac{1}{2\epsilon_{2}} \norm{\vect{a}\cdot\nabla w_{h}}^{2}_{0, K}
         +  \frac{\epsilon_{2}}{2} \norm{w_{h}}^{2}_{0, K}},
\end{split}
\end{equation}
where $\epsilon_{2} > 0$ but is otherwise arbitrary. Setting
$\epsilon_{2} = 2c\snorm{\vect{a}}_{W^{1}_{\infty}(K)}$
\begin{equation}
    R_{2}
   \le \frac{1}{4} h_{K} \norm{\vect{a}\cdot\nabla w_{h}}^{2}_{0, K}
         +  c h_{K} \snorm{\vect{a}}_{W^{1}_{\infty}(K)}^{2} \norm{w_{h}}^{2}_{0, K}.
\end{equation}
Setting $R_{3} =
         \norm{ h_{K} \Bar{\vect{a}}_{K}\cdot\nabla w_{h}}_{0, \pd K_{+}}
              \norm{a_{n}\brac{\Bar{w}_{h} - w_{h}}}_{0, \pd K_{+}}$
and using equation~\eqref{eqn:w_bound1} and Young's inequality,
\begin{equation}
\begin{split}
  R_{3}
\le&  c h^{-\frac{1}{2}}_{K}\norm{ h_{K} \Bar{\vect{a}}_{K}\cdot\nabla w_{h}}_{0, K}
              \norm{a_{n}\brac{\Bar{w}_{h} - w_{h}}}_{0, \pd K_{+}}
\\
  \le&
   c \brac{\norm{ h^{\frac{1}{2}}_{K} \vect{a}\cdot\nabla w_{h}}_{0, K}
    + \snorm{\vect{a}}_{W^{1}_{\infty}(K)} \norm{h^{\frac{1}{2}}_{K} w_{h}}_{0, K}}
              \norm{a_{n}\brac{\Bar{w}_{h} - w_{h}}}_{0, \pd K_{+}}
\\
  \le&
   \frac{c}{2\epsilon_{3}}\brac{\norm{ h^{\frac{1}{2}}_{K} \vect{a}\cdot\nabla w_{h}}_{0, K}
    + \snorm{\vect{a}}_{W^{1}_{\infty}(K)} \norm{h^{\frac{1}{2}}_{K} w_{h}}_{0, K}}^{2}
           + \frac{c\epsilon_{3}}{2}   \norm{a_{n}\brac{\Bar{w}_{h} - w_{h}}}_{0, \pd K_{+}}^{2}
\\
  \le&
   \frac{ch_{K}}{\epsilon_{3}}\brac{\norm{\vect{a}\cdot\nabla w_{h}}_{0, K}^{2}
    + \snorm{\vect{a}}_{W^{1}_{\infty}(K)}^{2} \norm{w_{h}}_{0, K}^{2}}
   + \frac{c\epsilon_{3}}{2} \norm{a_{n}\brac{\Bar{w}_{h} - w_{h}}}_{0, \pd K_{+}}^{2},
\end{split}
\end{equation}
where $\epsilon_{3} > 0$ but is otherwise arbitrary.
Setting $\epsilon_{3} = 4c$,
\begin{multline}
  R_{3}
  \le
   \frac{1}{4} h_{K}\norm{ \vect{a}\cdot\nabla w_{h}}_{0, K}^{2}
\\
    + c\brac{h_{K} \snorm{\vect{a}}_{W^{1}_{\infty}(K)}^{2} \norm{w_{h}}_{0, K}^{2}
           + \norm{\vect{a}}_{L^{\infty}(\Omega)} \norm{a^{\frac{1}{2}}_{n}\brac{\Bar{w}_{h} - w_{h}}}_{0, \pd K_{+}}^{2}}.
\end{multline}
Combining these results leads to
\begin{multline}
  \frac{1}{2} \sum_{K} h_{K} \norm{ \vect{a} \cdot \nabla w_{h}}^{2}_{0, K}
  \le
B_{A}\brac{\vect{z}_{h}, \vect{w}_{h}}
\\
+  c\max\brac{\norm{\vect{a}}_{L^{\infty}(\Omega)}, \frac{h_{K}\snorm{\vect{a}}_{W^{1}_{\infty}(\Omega)}^{2}}{\mu} } \brac{\mu\norm{w_{h}}^{2}_{0, \Omega}
+  \sum_{K}  \norm{a_{n}^{\frac{1}{2}}\brac{\Bar{w}_{h} - w_{h}}}^{2}_{0, \pd K}}.
\end{multline}
From the above result, the definition of the norm in~\eqref{eqn:stability_norm}
and coercivity~\eqref{eqn:coercivity}, the lemma follows straightforwardly
with
$c_{1} = c\max\brac{1, \snorm{\vect{a}}_{W^{1}_{\infty}(\Omega)}^{2} / \mu }$.
\end{proof}

\begin{proposition} \label{lemma:z_inequality}
For $\vect{z}_{h}$ which depends on $\vect{w}_{h}$ according to
equation~\eqref{eqn:z_function},
there exists a $c_{2} > 0$ such that for all
$\vect{w}_{h} \in W^{\star}_{h}$
\begin{equation}
  \unorm{\vect{z}_{h}}_{A} \le c_{2} \unorm{\vect{w}_{h}}_{A}.
\label{eqn:z_inequality}
\end{equation}
\end{proposition}
\begin{proof}
The components of $\unorm{\vect{z}_{h}}_{A}$ can be bounded term-by-term.
Using equation~\eqref{eqn:w_bound2},
\begin{equation}
\mu \norm{z_{h}}^{2}_{0, K}
  = \mu \norm{h_{K} \Bar{\vect{a}}_{K} \cdot \nabla w_{h}}^{2}_{0, K}
\le c \mu \norm{\vect{a}}^{2}_{L^{\infty}(\Omega)}\norm{w_{h}}^{2}_{0, K}.
\end{equation}
Using the inverse inequality~\eqref{eqn:inverse_gradient} and
equation~\eqref{eqn:w_bound1},
\begin{equation}
\begin{split}
h_{K}  \norm{\vect{a} \cdot \nabla z_{h}}^{2}_{0, K}
&\le c h^{-1}_{K} \norm{z_{h}}^{2}_{0, K}
\\
&\le
  c \brac{h_{K} \norm{\vect{a} \cdot \nabla w_{h}}^{2}_{0, K}
      + h_{K} \snorm{\vect{a}}^{2}_{W^{1}_{\infty}(K)}\norm{w_{h}}^{2}_{0, K}}.
\end{split}
\end{equation}
For the facet term,
\begin{multline}
\norm{a_{n}^{\frac{1}{2}} h_{K}\Bar{\vect{a}}_{K} \cdot \nabla w_{h}}^{2}_{0, \pd K}
\\
\le
ch_{K}\norm{a_{n}}_{L^{\infty}(\pd K)}\brac{h_{K}\norm{ \vect{a} \cdot \nabla w_{h}}^{2}_{0, K}
      +  h_{K} \snorm{\vect{a}}^{2}_{W^{1}_{\infty}(K)} \norm{w_{h}}^{2}_{0, K}}.
\end{multline}
This proves that
$\unorm{\vect{z}_{h}}_{A}
   \le c \max\brac{\norm{\vect{a}}_{L^{\infty}(\Omega)}, \brac{h\snorm{\vect{a}}^{2}_{W^{1}_{\infty}(\Omega)}/\mu}^{\frac{1}{2}}}
   \unorm{\vect{w}_{h}}_{A}$, with
$c_{2}
 = c\max\brac{\norm{\vect{a}}_{L^{\infty}(\Omega)}, \snorm{\vect{a}}_{W^{1}_{\infty}(\Omega)}/\mu^{\frac{1}{2}}}$.
\end{proof}

Setting $\vect{v}_{h} = c_{1}\vect{w}_{h} + \vect{z}_{h}$,
the preceding proposition also implies that
\begin{equation}
  \unorm{\vect{v}_{h}}_{A} = \unorm{c_{1}\vect{w}_{h} + \vect{z}_{h}}_{A}
      \le \brac{c_{1} + c_{2}} \unorm{\vect{w}_{h}}_{A}.
\label{eqn:v_inequality}
\end{equation}
Now, using the preceding two results, the demonstration of
inf-sup stability is straightforward.
\begin{lemma}[inf-sup stability] \label{lemma:inf_sup}
There exists a $\beta_{A} > 0$, which is independent of $h$,
such that for all $\vect{v}_{h} \in W^{\star}_{h}$
\begin{equation}
  \sup_{\vect{w}_{h} \in W^{\star}_{h}}
      \frac{B_{A}\brac{\vect{v}_{h}, \vect{w}_{h}}}{\unorm{\vect{w}_{h}}_{A}}
                       \ge \beta_{A} \unorm{\vect{v}_{h}}_{A}.
\label{eqn:inf_sup}
\end{equation}
\end{lemma}
\begin{proof}
For non-trivial $\vect{v}_{h} = c_{1}\vect{w}_{h} + \vect{z}_{h}$,
combining Lemma~\ref{lemma:w_norm_bound} and
Proposition~\ref{lemma:z_inequality}
(see also equation~\eqref{eqn:v_inequality}) yields
\begin{equation}
  \unorm{\vect{v}_{h}}_{A}
  \unorm{\vect{w}_{h}}_{A}
  \le
  \brac{c_{1} + c_{2}} \unorm{\vect{w}_{h}}_{A}^{2}
  \le
 2\brac{c_{1} + c_{2}} B_{A}\brac{\vect{v}_{h}, \vect{w}_{h}},
\end{equation}
which implies that for
$\beta_{A} = 1/ \brac{2\brac{c_{1} + c_{2}}}$,
there exists a function $\vect{w}_{h} \in W^{\star}_{h}$ such that
\begin{equation}
  \beta_{A} \unorm{\vect{v}_{h}}_{A} \le
      \frac{B_{A}\brac{\vect{v}_{h}, \vect{w}_{h}}}{\unorm{\vect{w}_{h}}_{A}}
  \quad \forall \vect{v}_{h} \in W^{\star}_{h}.
\end{equation}
This is satisfaction of the inf-sup condition.
\end{proof}

Note the dependence of $\beta_{A}$ on the problem data;
it becomes smaller as gradients in $\vect{a}$
become large and as $\mu$ becomes small.
In practice, this is a rather pessimistic scenario since often
additional $L^{2}$ control will be provided by the prescription of the solution
at inflow boundaries. Numerical experiments with $\mu = 0$ are usually
observed to be stable.
\subsection{Error analysis}
\label{sec:hyperbolic_error}
To reach an error estimate, continuity of the bilinear form with respect to
the norms defined in equations~\eqref{eqn:stability_norm}
and~\eqref{eqn:error_norm} is required.
It is the continuity requirement which necessitates the introduction
of the norm $\unorm{\cdot}_{A^{\prime}}$ in addition to the stability
norm $\unorm{\cdot}_{A}$.
\begin{lemma}[continuity] \label{lemma:hyperbolic_continuity}
There exists a $C_{A} > 0$, which is independent of $h$,
such that for all $\vect{w} \in W^{\star}(h)$
and for all $\vect{v}_{h} \in W_{h}^{\star}$
\begin{equation}
  |B_{A}\brac{\vect{w}, \vect{v}_{h}}|
     \le C_{A} \unorm{\vect{w}}_{A^{\prime}} \unorm{\vect{v}_{h}}_{A}.
\label{eqn:continuity}
\end{equation}
\end{lemma}
\begin{proof}
From the definition of the bilinear form:
\begin{equation}
\begin{split}
  |B_{A}\brac{\vect{w}, \vect{v}_{h}}|
    =&
    \left| \int_{\Omega} \mu w v_{h} \dif x
- \sum_{K} \int_{K} \vect{a} w \cdot  \nabla v_{h} \dif x
 - \sum_{K} \int_{\pd K_{+}} a_{n} w  \brac{\Bar{v}_{h} - v_{h}} \dif s
\right.
\\
&+ \left. \sum_{K}\int_{\pd K_{-}} a_{n} \Bar{w}  \brac{\Bar{v}_{h} - v_{h}} \dif s
  + \int_{\Gamma_{+}} a_{n} \Bar{w} \Bar{v}_{h}  \dif s
\right|
\\
  \le& \sum_{K} \norm{w}_{0, K}\brac{\mu \norm{v_{h}}_{0, K} + \norm{\vect{a} \cdot \nabla v_{h}}_{0, K}}
\\
&+ \sum_{K}  \norm{a^{\frac{1}{2}}_{n} w}_{0,\pd K_{+}} \norm{a^{\frac{1}{2}}_{n}\brac{\Bar{v}_{h} - v_{h}}}_{0,\pd K_{+}}
\\
&+ \sum_{K}  \norm{a^{\frac{1}{2}}_{n}\Bar{w}}_{0,\pd K_{-}} \norm{a^{\frac{1}{2}}_{n}\brac{\Bar{v}_{h} - v_{h}}}_{0,\pd K_{-}}
\\
   &+ \norm{a_{n}^{\frac{1}{2}}\Bar{w}}_{0,\Gamma_{+}} \norm{a_{n}^{\frac{1}{2}}\Bar{v}_{h}}_{0,\Gamma_{+}}.
\end{split}
\end{equation}
Now, bounding each term,
\begin{align}
  \sum_{K} \mu \norm{w}_{0,K}\norm{v_{h}}_{0,K}
      &\le \unorm{\vect{w}}_{A^{\prime}} \unorm{\vect{v}_{h}}_{A},
\\
  \sum_{K} h_{K}^{-\frac{1}{2}}\norm{w}_{0,K} h_{K}^{\frac{1}{2}}\norm{\vect{a}\cdot \nabla v_{h}}_{0,K}
      &\le \unorm{\vect{w}}_{A^{\prime}} \unorm{\vect{v}_{h}}_{A},
\\
  \sum_{K} \norm{a_{n}^{\frac{1}{2}} w}_{0,\pd K_{+}} \norm{a^{\frac{1}{2}}_{n}\brac{\Bar{v}_{h} - v_{h}}}_{0,\pd K_{+}}
    &\le \unorm{\vect{w}}_{A^{\prime}} \unorm{\vect{v}_{h}}_{A},
\\
  \sum_{K} \norm{a^{\frac{1}{2}}_{n} \Bar{w}}_{0,\pd K_{-}} \norm{a^{\frac{1}{2}}_{n}\brac{\Bar{v}_{h} - v_{h}}}_{0,\pd K_{-}}
      &\le \unorm{\vect{w}}_{A^{\prime}} \unorm{\vect{v}_{h}}_{A},
\\
    \norm{a_{n}^{\frac{1}{2}}\Bar{w}}_{0,\Gamma_{+}} \norm{a_{n}^{\frac{1}{2}}\Bar{v}_{h}}_{0,\Gamma_{+}}
  &\le \unorm{\vect{w}}_{A^{\prime}} \unorm{\vect{v}_{h}}_{A}.
\end{align}
Summation of these bounds leads to the result, and demonstrates that $C_{A} = 1$.
\end{proof}

The necessary results are now in place in to prove convergence of the method.
\begin{lemma}[convergence]
For the case $\kappa = 0$, if $\vect{u} = \brac{u, u}$, where $u$ solves
equation~\eqref{eqn:strong}
and $\vect{u}_{h}$ is the solution to the finite element
problem~\eqref{eqn:fe_problem}, then
\begin{equation}
  \unorm{\vect{u} - \vect{u}_{h}} _{A}
      \le \brac{1 + \frac{C_{A}}{\beta_{A}}} \inf_{\vect{v}_{h} \in W^{\star}}
         \unorm{\vect{u} - \vect{v}_{h}}_{A^{\prime}}.
\label{eqn:convergence}
\end{equation}
\end{lemma}
\begin{proof}
From inf-sup stability (Lemma~\ref{lemma:inf_sup}),
consistency (Lemma~\ref{lemma:consistency}) and
continuity of the bilinear form (Lemma~\ref{lemma:hyperbolic_continuity}):
\begin{equation}
\begin{split}
  \beta_{A}\unorm{\vect{u}_{h} - \vect{w}_{h}}_{A}
      &\le
      \sup_{\vect{v}_{h} \in W_{h}^{\star}}
        \frac{B_{A}\brac{\vect{u}_{h} - \vect{w}_{h},
                       \vect{v}_{h}}}{\unorm{\vect{v}_{h}}_{A}}
    =
      \sup_{\vect{v}_{h} \in W_{h}^{\star}}
        \frac{B_{A}\brac{\vect{u} - \vect{w}_{h},
                    \vect{v}_{h}}}{\unorm{\vect{v}_{h}}_{A}}
\\
    &\le
      C_{A} \sup_{\vect{v}_{h} \in W_{h}^{\star}}
        \frac{\unorm{\vect{u} - \vect{w}_{h}}_{A^{\prime}}
                \unorm{\vect{v}_{h}}_{A}}{\unorm{\vect{v}_{h}}_{A}}
  = C_{A} \unorm{\vect{u} - \vect{w}_{h}}_{A^{\prime}}.
\end{split}
\end{equation}
Application of the triangle inequality
\begin{equation}
  \unorm{\vect{u} - \vect{u}_{h}}_{A}
        \le \unorm{\vect{u} - \vect{w}_{h}}_{A}  + \unorm{\vect{w}_{h} - \vect{u}_{h}}_{A}
\end{equation}
and $\unorm{\vect{v}}_{A} \le \unorm{\vect{v}}_{A^{\prime}}$ yields the result.
\end{proof}
\begin{lemma}[best approximation]
For the case $\kappa = 0$, if $u \in H^{k+1}\brac{\Omega}$
solves equation~\eqref{eqn:strong} and $\vect{u} = \brac{u, u}$,
and $\vect{u}_{h}$ is the solution to the finite element
problem~\eqref{eqn:fe_problem}, then
there exists a $c_{\mu, \vect{a}} > 0$ such that
\begin{equation}
  \unorm{\vect{u} - \vect{u}_{h}}_{A}
      \le c_{\mu, \vect{a}} h^{k+\frac{1}{2}} \norm{u}_{k+1, \Omega}
\label{eqn:hyperbolic_error_1}
\end{equation}
and
\begin{equation}
  \norm{u - u_{h}}_{0, \Omega}
      \le c_{\mu, \vect{a}} h^{k+\frac{1}{2}} \norm{u}_{k+1, \Omega}.
\label{eqn:hyperbolic_error_2}
\end{equation}
\end{lemma}
\begin{proof}
The continuous interpolant of $\vect{u}$ is denoted by
$\mathcal{I}_{h}\vect{u} = \brac{\mathcal{I}_{h}u, \Bar{\mathcal{I}}_{h}u}$,
where $\mathcal{I}_{h}u \in W_{h} \cap C\brac{\Bar{\Omega}}$
and
$\Bar{\mathcal{I}}_{h}u = \left. \mathcal{I}_{h}u\right|_{\Gamma^{0}}$,
which is contained in~$\Bar{W}_{h}$.
The standard interpolation estimate reads:
\begin{equation}
  \norm{u - \mathcal{I}_{h}u}_{m, K} \le c h_{K}^{k+1-m} \snorm{u}_{k+1, K}.
\label{eqn:interpolation_estimate}
\end{equation}
Bounding each term in $\unorm{\vect{u} - \mathcal{I}_{h}\vect{u}}_{A^{\prime}}$,
\begin{align}
  \norm{u - \mathcal{I}_{h}u}^{2}_{0,K}
        &\le c h_{K}^{2\brac{k+1}} \snorm{u}^{2}_{k+1, K},
\\
  h_{K}\norm{\vect{a} \cdot \nabla\brac{u - \mathcal{I}_{h}u}}^{2}_{0,K}
        &\le c h_{K}^{2k+1} \snorm{u}^{2}_{k+1, K},
\\
  \norm{\brac{u - \Bar{\mathcal{I}}_{h}u} - \brac{u - \mathcal{I}_{h}u}}^{2}_{0,\pd K}
        &= 0,
\\
  h_{K}^{-1} \norm{u - \mathcal{I}_{h}u}^{2}_{0,K}
        &\le c h_{K}^{2k+1} \snorm{u}^{2}_{k+1, K},
\end{align}
\begin{multline}
  \norm{u - \Bar{\mathcal{I}}_{h}u}^{2}_{0,\pd K}
 =
  \norm{u - \mathcal{I}_{h}u}^{2}_{0,\pd K}
\\
    \le c\brac{h_{K}^{-1}\norm{u - \mathcal{I}_{h}u}^{2}_{0, K} + h_{K}\snorm{u - \mathcal{I}_{h}u}^{2}_{1, K}   }
    \le c h_{K}^{2k +1} \snorm{u}^{2}_{k+1, K}.
\end{multline}
Using these results and equation~\eqref{eqn:convergence}
leads to the convergence estimates.
\end{proof}
\section{Analysis in the diffusive limit}
The diffusive limit ($\vect{a} = \vect{0}$, $\mu =0$) is now considered, in
which case the  bilinear form is given by equation~\eqref{eqn:B_diffusive}.
The analysis of the diffusive case is considerably simpler than for the hyperbolic
case since stability can be demonstrated via
coercivity of the bilinear form.
Analysis tools and results which are typically
used in the analysis of discontinuous Galerkin methods for elliptic
problems~\citep{arnold:2002}
are leveraged against this problem.

To ease the notational burden, the case of homogeneous Dirichlet
boundary conditions on $\Gamma$ is considered. The extended
function spaces
\begin{align}
  W(h)      &= W_{h} + H^{2}\brac{\Omega} \cap H^{1}_{0}\brac{\Omega} , \\
 \Bar{W}(h) &= \bar{W}_{h} + H^{3/2}_{0}\brac{\Gamma^{0}},
\end{align}
will be used,
where $H^{3/2}_{0}\brac{\Gamma^{0}}$ denotes the trace space
of $H^{2}\brac{\Omega} \cap H^{1}_{0}\brac{\Omega}$ on facets~$\Gamma^{0}$.

As for the hyperbolic case, two norms on
$W^{\star}\brac{h} = W(h) \times \Bar{W}(h)$ are introduced for the examination
of stability and continuity. The `stability' norm reads
\begin{equation}
  \unorm{\vect{v}}^{2}_{D} =
\sum_{K}\kappa \norm{\nabla v}^{2}_{0, K}
+ \sum_{K} \frac{\alpha \kappa}{h_{K}}
             \norm{\Bar{v} - v}^{2}_{0, \pd K},
\end{equation}
and the `continuity' norm reads
\begin{equation}
  \unorm{\vect{v}}^{2}_{D^{\prime}}
  = \unorm{\vect{v}}^{2}_{D}
    +  \sum_{K} \frac{h_{K}^{2} \kappa}{\alpha}\snorm{v}_{2, K}^{2}.
\end{equation}
It is clear from the definitions that
$\unorm{\vect{v}}^{2}_{D} \le \unorm{\vect{v}}^{2}_{D^{\prime}}$,
but there also exists a constant $c > 0$ such that for all $\vect{v}_{h} \in W^{\star}_{h}$
\begin{equation}
  \unorm{\vect{v}_{h}}_{D^{\prime}} \le c \brac{1 + \alpha^{-1}}\unorm{\vect{v}_{h}}_{D},
\label{eqn:diff_norm_relationship}
\end{equation}
since from equation~\eqref{eqn:dg_trace_orig} it follows that
\begin{equation}
  h^{2}_{K} \snorm{v_{h}}_{2, K}^{2} \le c \norm{\nabla v_{h} }^{2}_{0, K}
\quad \forall v_{h} \in W_{h}.
\end{equation}
Therefore, the  norms $\unorm{\cdot}_{D}$ and $\unorm{\cdot}_{D^{\prime}}$ are equivalent
on the finite element space~$W^{\star}_{h}$.

To demonstrate that $\unorm{\cdot}^{2}_{D}$ and $\unorm{\cdot}^{2}_{D^{\prime}}$
do constitute norms, first recall that for a facet $F$
\begin{equation}
\begin{split}
 \sum_{F} \norm{v_{+} - v_{-}}_{0, F}
=& \sum_{F} \norm{\brac{v_{+} - \Bar{v}} - \brac{v_{-} - \Bar{v}}}_{0, F}
\\
\le& \sum_{F} \norm{\Bar{v} - v_{+}}_{0, F} + \norm{\Bar{v} - v_{-}}_{0, F}
\\
=& \sum_{K} \norm{\Bar{v} - v}_{0, \pd K}.
\end{split}
\end{equation}
Denoting the average size of two cells sharing a facet by~$h_{F}$,
\begin{equation}
\norm{\kappa^{\frac{1}{2}} v}^{2}_{0,\Omega} \le
c_{1} \brac{\sum_{K} \kappa \norm{\nabla v}^{2}_{0, K}
+ \sum_{F} \frac{\kappa}{h_{F}}
    \norm{v^{+} - v^{-} }^{2}_{0, F}
}
\le c_{1}\brac{1+\alpha^{-1}} \unorm{\vect{v}}^{2}_{D},
\end{equation}
where the first inequality is a standard result
(See \citet[Lemma 2.1]{arnold:1982} and \citet[Lemma 3.45]{ern:book}).
Hence, $\unorm{\cdot}_{D}$ and $\unorm{\cdot}_{D^{\prime}}$
constitute norms.
\subsection{Stability}
Before proceeding to coercivity of the bilinear form,
an intermediate result is presented.
\begin{proposition} \label{lemma:dg_elliptic_penalty_bound}
There exists a constant $c > 0$ such that for
any $\epsilon > 0$ and  all $\vect{v}_{h} \in W^{\star}_{h}$
\begin{equation}
    \left| 2 \int_{\pd K} \kappa \nabla v_{h} \cdot \vect{n}
                \brac{\Bar{v}_{h} - v_{h}} \dif s \right|
  \le
  \epsilon c \kappa \norm{\nabla v_{h}}_{0, K}^{2}
 +
  \frac{\kappa}{\epsilon h_{K}} \norm{\Bar{v}_{h} - v_{h}}^{2}_{0, \pd K}.
\label{eqn:dg_inequality}
\end{equation}
\end{proposition}
\begin{proof}
Applying to the term $\int_{\pd K} \nabla v_{h} \cdot \vect{n} \brac{\Bar{v}_{h} - v_{h}}  \dif s$
the Cauchy-Schwarz inequality, the inverse estimates
and Young's inequality,
\begin{equation}
\begin{split}
\left| 2 \int_{\pd K} \kappa \nabla v_{h} \cdot  \brac{\Bar{v}_{h} - v_{h}} \vect{n} \dif s \right|
  \le& 2 \kappa h_{K}\norm{\nabla v_{h} \cdot \vect{n}}_{0, \pd K} h^{-1}_{K}\norm{\Bar{v}_{h} - v_{h}}_{0, \pd K}
\\
  \le&  h_{K} \kappa \epsilon  \norm{\nabla v_{h} \cdot \vect{n}}^{2}_{0, \pd K}
            + \frac{\kappa}{ \epsilon h_{K}} \norm{\Bar{v}_{h} - v_{h}}^{2}_{0, \pd K}
\\
  \le&  \epsilon c \kappa \brac{\snorm{v_{h}}_{1, K}^{2} + h^{2}_{K} \snorm{v_{h}}_{2, K}^{2}}
          + \frac{\kappa}{\epsilon h_{K} } \norm{\Bar{v}_{h} - v_{h}}_{0, \pd K}^{2}
\\
  \le&  \epsilon c\kappa  \norm{\nabla v_{h}}^{2}_{0, K}  + \frac{\kappa}{\epsilon h_{K} } \norm{\Bar{v}_{h} - v_{h}}^{2}_{0, \pd K},
\end{split}
\end{equation}
which complete the proof.
\end{proof}
\begin{lemma}[coercivity] \label{lemma:diff_coercivity}
There exists a $\beta_{D} > 0$, independent of $h$,
and a constant
$\alpha_{0} > 0$ such that for $\alpha > \alpha_{0}$ and
for all $\vect{v}_{h} \in W^{\star}_{h}$
\begin{equation}
  B_{D}\brac{\vect{v}_{h}, \vect{v}_{h}} \ge \beta_{D} \unorm{\vect{v}_{h}}^{2}_{D},
\label{eqn:diff_coercivity_general}
\end{equation}
and there exists an $\alpha_{1} > \alpha_{0}$ such that for all
$\vect{v}_{h} \in W^{\star}_{h}$
\begin{equation}
  B_{D}\brac{\vect{v}_{h}, \vect{v}_{h}} \ge  \frac{1}{2} \unorm{\vect{v}_{h}}^{2}_{D}.
\label{eqn:diff_coercivity_specific}
\end{equation}
\end{lemma}
\begin{proof}
Setting $\vect{w}_{h} = \vect{v}_{h}$ in the bilinear form for the
diffusive limit case~\eqref{eqn:B_diffusive},
\begin{multline}
  B_{D}\brac{\vect{v}, \vect{v}}
     =
\sum_{K} \kappa \norm{ \nabla v_{h}}^{2}_{0, K}
    + 2 \sum_{K} \int_{\pd K} \kappa \nabla v_{h} \cdot \vect{n} \brac{\Bar{v}_{h} - v_{h}} \dif s
\\
+ \sum_{K} \frac{\alpha \kappa}{h_{K}} \norm{\Bar{v}_{h} - v_{h}}^{2}_{0,\pd K}.
\label{eqn:diff_coercivity_proof_step}
\end{multline}
Using Proposition~\ref{lemma:dg_elliptic_penalty_bound} to bound the term
$\sum_{K} \int_{\pd K} \kappa \nabla v_{h} \cdot \vect{n} \brac{\Bar{v}_{h} - v_{h}} \dif s$,
\begin{equation}
  B_{D}\brac{\vect{v}_{h}, \vect{v}_{h}}
     \ge
\sum_{K} \brac{1 - \epsilon c}\kappa \norm{\nabla v_{h}}^{2}_{0, K}
+ \sum_{K} \brac{\alpha - \frac{1}{\epsilon}} \frac{\kappa}{h_{K}} \norm{\Bar{v}_{h} - v_{h}}^{2}_{0,\pd K}.
\end{equation}
Setting $\epsilon = 1/\delta c$, where $\delta > 1$ but is otherwise arbitrary,
\begin{equation}
  B_{D}\brac{\vect{v}_{h}, \vect{v}_{h}}
     \ge
\sum_{K} \brac{1 - \frac{1}{\delta}}\kappa \norm{\nabla v_{h}}^{2}_{0, K}
+ \sum_{K} \brac{\alpha - \delta c}  \frac{\kappa}{h_{K}} \norm{\Bar{v}_{h} - v_{h}}^{2}_{0,\pd K}.
\end{equation}
This proves equation~\eqref{eqn:diff_coercivity_general} and demonstrates that
$\alpha_{0} = c$. Setting $\delta = 2$,
\begin{equation}
  B_{D}\brac{\vect{v}_{h}, \vect{v}_{h}}
     \ge
\sum_{K} \frac{1}{2}\kappa \norm{ \nabla v_{h}}^{2}_{0, K}
+ \sum_{K} \brac{\alpha - 2 c}  \frac{\kappa}{h_{K}} \norm{\Bar{v}_{h} - v_{h}}^{2}_{0,\pd K},
\end{equation}
which proves \eqref{eqn:diff_coercivity_specific} when $\alpha_{1} = 1/2 + 2c$.
\end{proof}

The proof to Lemma~\ref{lemma:diff_coercivity} demonstrates that stability
is enhanced for a larger penalty parameter. Stability demands that
$\alpha > c$, and when this is satisfied $\delta$ can be chosen such that
$\beta_{D}$ approaches zero as $\alpha$ approaches $\alpha_{0}$, and such that
$\beta_{D}$ approaches one as $\alpha$ becomes much larger than~$\alpha_{0}$.

\subsection{Error analysis}
The error analysis proceeds in a straightforward manner now that the stability
result is in place.
\begin{lemma}[continuity]
There exists a $C_{D} > 0$, independent of $h$, such that
for all $\vect{w} \in W^{\star}(h)$ and for all $\vect{v} \in W^{\star}(h)$
\begin{equation}
  |B_{D}\brac{\vect{w}, \vect{v}}|
      \le C_{D} \unorm{\vect{w}}_{D^{\prime}} \unorm{\vect{v}}_{D^{\prime}}.
\end{equation}
\end{lemma}
\begin{proof}
From the definition of the bilinear form,
\begin{multline}
  |B_{D}\brac{\vect{w}, \vect{v}_{h}}|
    \le \sum_{K} \kappa \norm{\nabla w}_{0, K} \norm{\nabla v_{h}}_{0, K}
  + \sum_{K} \kappa \norm{\nabla w \cdot \vect{n}}_{0, \pd  K} \norm{\Bar{v}_{h} - v_{h}}_{0, \pd K}
\\
  + \sum_{K} \kappa \norm{\Bar{w} - w}_{0, \pd K} \norm{\nabla v_{h} \cdot \vect{n}}_{0, \pd K}
  + \sum_{K} \frac{\alpha\kappa}{h_{K}} \norm{\Bar{w} - w}_{0, \pd K}
                     \norm{\Bar{v}_{h} - v_{h}}_{0, \pd K}.
\end{multline}
Each term can be bounded appropriately,
\begin{equation}
  \sum_{K} \kappa \norm{\nabla w}_{0, K} \norm{\nabla v_{h}}_{0, K}
        \le \unorm{\vect{w}}_{D^{\prime}} \sum_{K} \kappa^{\frac{1}{2}}\norm{\nabla v_{h}}_{0, K},
\end{equation}
\begin{multline}
  \sum_{K} \kappa  \norm{\nabla w  \cdot \vect{n}}_{0, \pd  K} \norm{\Bar{v}_{h} - v_{h}}_{0, \pd K}
  \le \sum_{K} c\brac{h^{-\frac{1}{2}}_{K}\snorm{w}_{1, K} + h^{\frac{1}{2}}_{K} \snorm{w}_{2, K} } \norm{\Bar{v}_{h} - v_{h}}_{0, \pd K}
\\
   \le c \sum_{K} \max\brac{1, \alpha^{-\frac{1}{2}}} \unorm{\vect{w}}_{D^{\prime}} \sum_{K}
                    \brac{\frac{\alpha\kappa}{h_{K}}}^{\frac{1}{2}}
                     \norm{\Bar{v}_{h} - v_{h}}_{0, \pd K},
\end{multline}
\begin{multline}
  \sum_{K} \kappa \norm{\Bar{w} - w}_{0, \pd K} \norm{\nabla v_{h} \cdot \vect{n}}_{0, \pd K}
\le
  c \sum_{K} \norm{\Bar{w} - w}_{0, \pd K} \brac{h^{-\frac{1}{2}}_{K}\snorm{v}_{1, K} + h^{\frac{1}{2}}_{K} \snorm{v}_{2, K} }
\\
\le c \max\brac{1, \alpha^{-\frac{1}{2}}} \unorm{\vect{w}}_{D^{\prime}}
  \sum_{K} \brac{\frac{\kappa}{\alpha}}^{\frac{1}{2}} \brac{\snorm{v}_{1, K} + h_{K} \snorm{v}_{2, K} },
\end{multline}
\begin{equation}
  \sum_{K} \frac{\alpha\kappa}{h_{K}} \norm{\Bar{w} - w}_{0, \pd K}
                     \norm{\Bar{v} - v}_{0, \pd K}
        \le \unorm{\vect{w}}_{D^{\prime}}
                 \sum_{K}\brac{\frac{\alpha\kappa}{h_{K}}}^{\frac{1}{2}}
                 \norm{\Bar{v} - v}_{0, \pd K}.
\end{equation}
Summing these inequalities shows that the bilinear form is continuous with
respect to~$\unorm{\cdot}_{D^\prime}$, with $C_{A} = c \max\brac{1, \alpha^{-\frac{1}{2}}}$.
\end{proof}

The penalty term $\alpha$ is usually taken to be greater than one, in which case $C_{D} = c$.

\begin{lemma}[convergence]
If $\alpha$ is chosen suitably large such that the bilinear form is coercive,
then if $u$ solves equation~\eqref{eqn:strong} and $\vect{u} = \brac{u, u}$,
and $\vect{u}_{h}$ is the solution to equation~\eqref{eqn:fe_problem}
for the case $\mu = 0$, $\vect{a} = \vect{0}$ and $\kappa > 0$,
then
\begin{equation}
  \unorm{\vect{u} - \vect{u}_{h}}_{D}
      \le \brac{1 + \frac{\brac{1+c\alpha^{-1}}C_{D}}{\beta_{D}}} \inf_{\vect{w}_{h} \in W^{\star}_{h}}
         \unorm{\vect{u} - \vect{w}_{h}}_{D^{\prime}}.
\end{equation}
\end{lemma}
\begin{proof}
Using coercivity, consistency and continuity:
\begin{equation}
\begin{split}
  \beta_{D} \unorm{\vect{u}_{h} - \vect{w}_{h}}^{2}_{D}
    &\le B_{D}\brac{\vect{u}_{h} - \vect{w}_{h}, \vect{u}_{h} - \vect{w}_{h}}
\\
    &= B_{D}\brac{\vect{u} - \vect{w}_{h}, \vect{u}_{h} - \vect{w}_{h}}
\\
  &\le C_{D} \unorm{\vect{u} - \vect{w}_{h}}_{D^{\prime}}
  \unorm{\vect{u}_{h} - \vect{w}_{h}}_{D^{\prime}},
\end{split}
\end{equation}
and then exploiting
$\unorm{\vect{v}_{h}}_{D^{\prime}} \le \brac{1 + c \alpha^{-1}} \unorm{\vect{v}_{h}}_{D}$ (see
equation~\eqref{eqn:diff_norm_relationship}),
\begin{equation}
    \unorm{\vect{u}_{h} - \vect{w}_{h}}_{D}
  \le \beta_{D}^{-1} C_{D} \brac{1 + c\alpha^{-1}} \unorm{\vect{u} - \vect{w}_{h}}_{D^{\prime}},
\end{equation}
which followed by the application of the triangle inequality yields the desired result.
\end{proof}

\begin{lemma}[best approximation]
For the case $\mu = 0$, $\vect{a} = \vect{0}$ and $\kappa > 0$,
if $u \in H^{k+1}\brac{\Omega}$
solves equation~\eqref{eqn:strong} and $\vect{u} = \brac{u, u}$,
and $\vect{u}_{h}$ is the solution to the finite element
problem~\eqref{eqn:fe_problem}, and $\alpha$ is chosen
such that the bilinear form is coercive, then there a exists a
$c_{\alpha} > 0$ such that

\begin{equation}
  \unorm{\vect{u} - \vect{u}_{h}}_{D}  \le c_{\alpha} h^{k} \norm{u}_{k+1, \Omega}
\label{eqn:diff_conv_rate}
\end{equation}
and
\begin{equation}
  \norm{u - u_{h}}_{0, \Omega}  \le c_{\alpha} h^{k+1} \norm{u}_{k+1, \Omega}.
\end{equation}
\end{lemma}
\begin{proof}
The first estimate follows directly from the
standard interpolation estimate for the continuous interpolant
$\mathcal{I}_{h}\vect{u} = \brac{\mathcal{I}_{h}u, \Bar{\mathcal{I}}_{h}u}$,
where again $\mathcal{I}_{h}u \in W_{h} \cap C\brac{\Bar{\Omega}}$
and  $\Bar{\mathcal{I}}_{h}u = \left. \mathcal{I}_{h}u\right|_{\Gamma^{0}}$,
which is an element of~$\Bar{W}_{h}$.
Applying the standard interpolation estimate~\eqref{eqn:interpolation_estimate}
to $\unorm{\vect{u} - \mathcal{I}_{h}\vect{u}}_{D^{\prime}}$,
\begin{align}
  \norm{\nabla\brac{u - \mathcal{I}_{h} u}}^{2}_{0, K}
                &\le c h^{2k} \snorm{u}^{2}_{k+1, K},
\\
  \norm{\brac{u - \Bar{\mathcal{I}}_{h} u}-\brac{u - \mathcal{I}_{h} u}}_{0, \pd K}^{2}
                    &= 0,
\\
  h^{2}_{K} \snorm{u - \mathcal{I} u_{h}}^{2}_{2, K}
                    &\le ch^{2k} \snorm{u}^{2}_{k+1, K}.
\end{align}
Using these inequalities leads to equation~\eqref{eqn:diff_conv_rate}.
The $L^{2}$ estimate follows from the usual duality arguments.
Owing to adjoint consistency of the method (since the bilinear form
is symmetric), if $\vect{w} \in H^{2}(\Omega) \cap H^{1}_{0}(\Omega)$
is the solution to the dual problem
\begin{equation}
  B_{D}\brac{\vect{v}, \vect{w}}
       = \int_{\Omega} \brac{u - u_{h}} \cdot v \dif x
               \quad \forall \vect{v} \in W^{\star}(h),
\end{equation}
and $\vect{w}_{I} \in W^{\star}_{h}$ is a suitable interpolant of $\vect{w}$,
then from consistency and continuity of the bilinear form,
and the estimate in~\eqref{eqn:diff_conv_rate}, it follows that
\begin{equation}
\begin{split}
  \norm{u - u_{h}}^{2}_{0, \Omega}
      &=   B_{D}\brac{\vect{u} - \vect{u}_{h}, \vect{w}}
\\
      &=   B_{D}\brac{\vect{u} - \vect{u}_{h}, \vect{w} - \vect{w}_{I}}
\\
      &\le C_{D} \unorm{\vect{u} - \vect{u}_{h}}_{D^{\prime}} \unorm{\vect{w} - \vect{w}_{I}}_{D^{\prime}}
\\
      &\le c_{\alpha} h \norm{w}_{2, \Omega} \unorm{\vect{u} - \vect{u}_{h}}_{D^{\prime}}.
\end{split}
\end{equation}
Finally, using the elliptic regularity estimate
$\norm{w}_{2, \Omega} \le c_{\alpha} \norm{u - u_{h}}_{0, \Omega}$ leads to
\begin{equation}
  \norm{u - u_{h}}_{0, \Omega}
      \le c_{\alpha} h \unorm{\vect{u} - \vect{u}_{h}}_{D^{\prime}}.
\end{equation}
The $L^{2}$ error estimate follows trivially.
\end{proof}

\section{Observed stability and convergence properties}
\label{sec:examples}
Some numerical examples are now presented to examine stability and convergence
properties of the method. In all examples, the interface functions
are chosen to be continuous
everywhere ($l =1$ in equation~\eqref{eqn:interface_space}), so the number of
global degrees of freedom is the same as for a continuous Galerkin method on the
same mesh. When computing the error for cases using
polynomial basis order $k$, the source term and the exact solution are
interpolated on the same mesh but using Lagrange elements of order~$k+6$.
Likewise, if the field $\vect{a}$ does not come from a finite element space
it is interpolated using order $k+6$ Lagrange elements.
Exact integration is performed for all terms.
All meshes are uniform and
the measure of the cell size $h_{K}$
is set to two times the circumradius of cell~$K$.

The computer code used for all examples in this section is freely available
in the supporting material~\citep{wells:2009} under a GNU Public License.
The necessary low-level computer code specific to this problem has been
generated automatically from a high-level scripted input language
using freely available tools from the FEniCS
Project~\citep{kirby:2006,oelgaard:2008,oelgaard:2008b,logg:2009}.
The computer input resembles closely the mathematical
notation and abstractions used in this work to describe the method.
Particular advantage is taken of automation
developments for methods that involve facet
integration~\citep{oelgaard:2008}.
\subsection{Hyperbolic problem}
Consider the domain
$\Omega = \brac{-1, 1}^{2}$, with $\mu = 1$,
$\vect{a} = \brac{0.8, 0.6}$,  $\kappa = 0$ and $u = 1$ on $\Gamma_{-}$.
The source term $f$ is chosen such that
\begin{equation}
  u = 1+ \sin\brac{ \pi \brac{1+x}\brac{1+y}^{2}/8}
\end{equation}
is the analytical solution to equation~\eqref{eqn:strong}.
This example has been considered previously
for discontinuous Galerkin methods
by \citet{bey:1996}
and \citet{houston:2000}.

The computed error $\unorm{\vect{u} - \vect{u}_{h}}_{A}$
is presented in Figure~\ref{fig:hyperbolic_h_unorm}
for $h$-refinement with various polynomial orders.
As predicted by the analysis, the
observed converge rate is~$k + 1/2$.
\begin{figure}
  \center\includegraphics[width=0.75\textwidth]{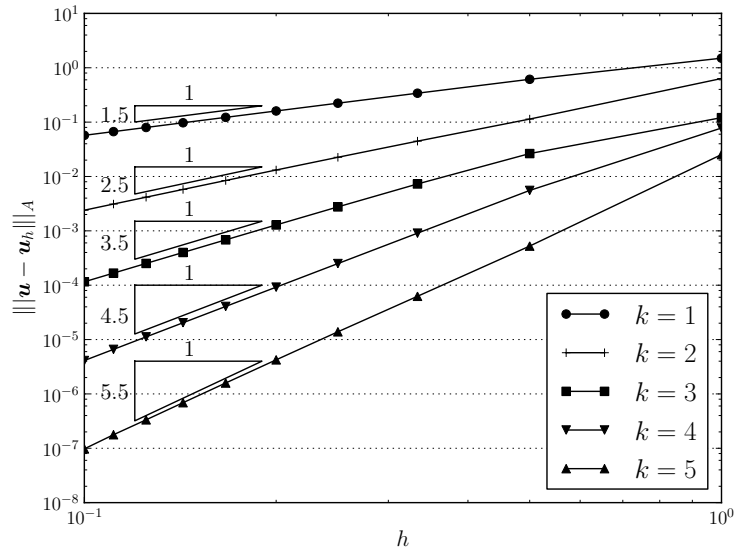}
\caption{Convergence for the hyperbolic case with $h$-refinement for
various polynomial orders in $\unorm{\cdot}_{A}$.}
\label{fig:hyperbolic_h_unorm}
\end{figure}
For all polynomial orders the method converges robustly.
\subsection{Elliptic problem}
A problem on the domain $\Omega = \brac{-1, 1}^{2}$ is now considered,
with $\mu = 0$, $\vect{a} = \brac{0, 0}$ and $\kappa = 1$.
The source term $f$ is selected such that
\begin{equation}
  u = \sin\brac{\pi x}\sin\brac{\pi y}
\label{eqn:elliptic_exact}
\end{equation}
is the analytical solution to equation~\eqref{eqn:strong}. The value
of the penalty parameter is stated for each considered case.

The computed errors in the $L^{2}$ norm for $h$-refinement with elements
of varying polynomial order and $\alpha = 5$ are shown in
Figure~\ref{fig:elliptic_h_5}.
\begin{figure}
  \center\includegraphics[width=0.75\textwidth]{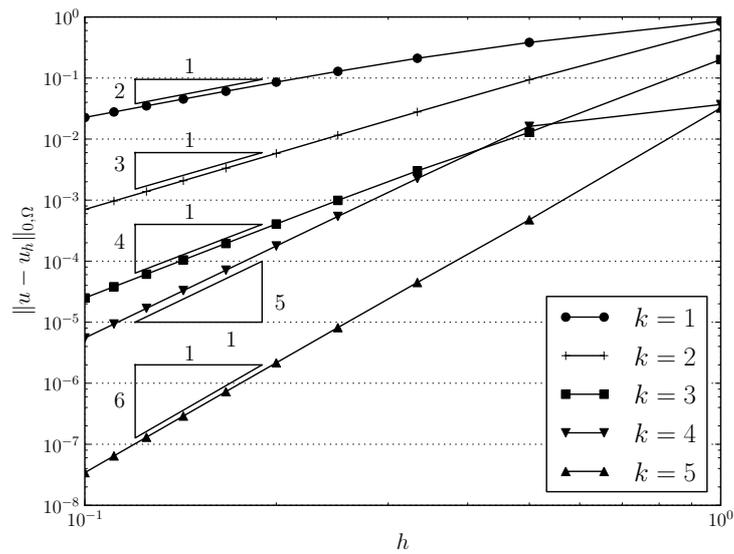}
\caption{Convergence for the elliptic case in $L^{2}$ with $h$-refinement
for various polynomial orders and $\alpha = 5$.}
\label{fig:elliptic_h_5}
\end{figure}
In all cases, the predicted $k+1$ order of convergence is observed.
The computed results for $\alpha = 6$ are shown in
Figure~\ref{fig:elliptic_h_6}, in which the convergence for the $k=2$
case is somewhat erratic.
\begin{figure}
  \center\includegraphics[width=0.75\textwidth]{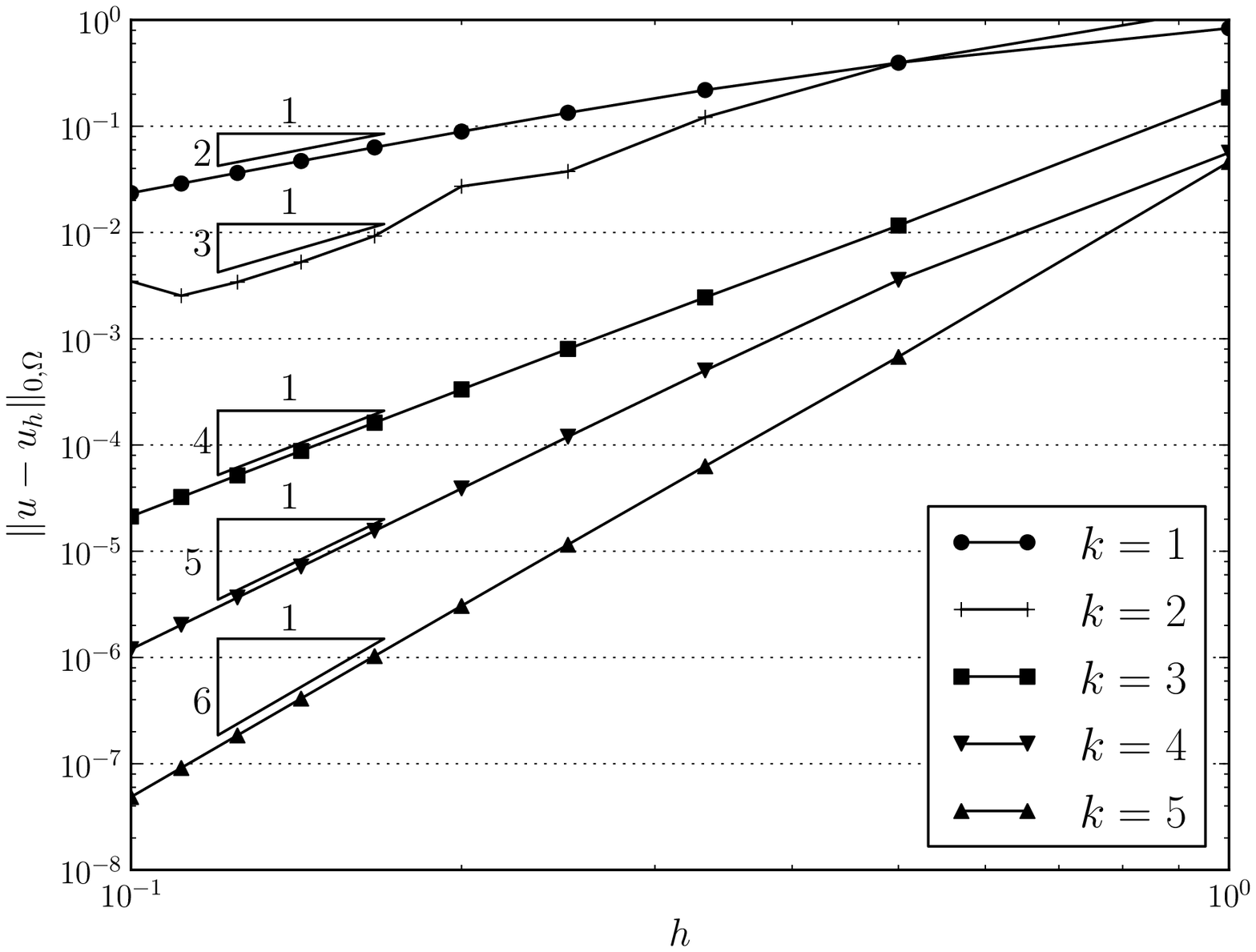}
\caption{Convergence for the elliptic case  in $L^{2}$ with $h$-refinement
for various polynomial orders and~$\alpha = 6$.}
\label{fig:elliptic_h_6}
\end{figure}
Using $\alpha = 4k^{2}$, since the penalty parameter for the interior penalty
method usually needs to be increased with increasing polynomial order,
reliable convergence behaviour at the predicted rate
is recovered, as can be seen in
Figure~\ref{fig:elliptic_h_4kk}.
\begin{figure}
  \center\includegraphics[width=0.75\textwidth]{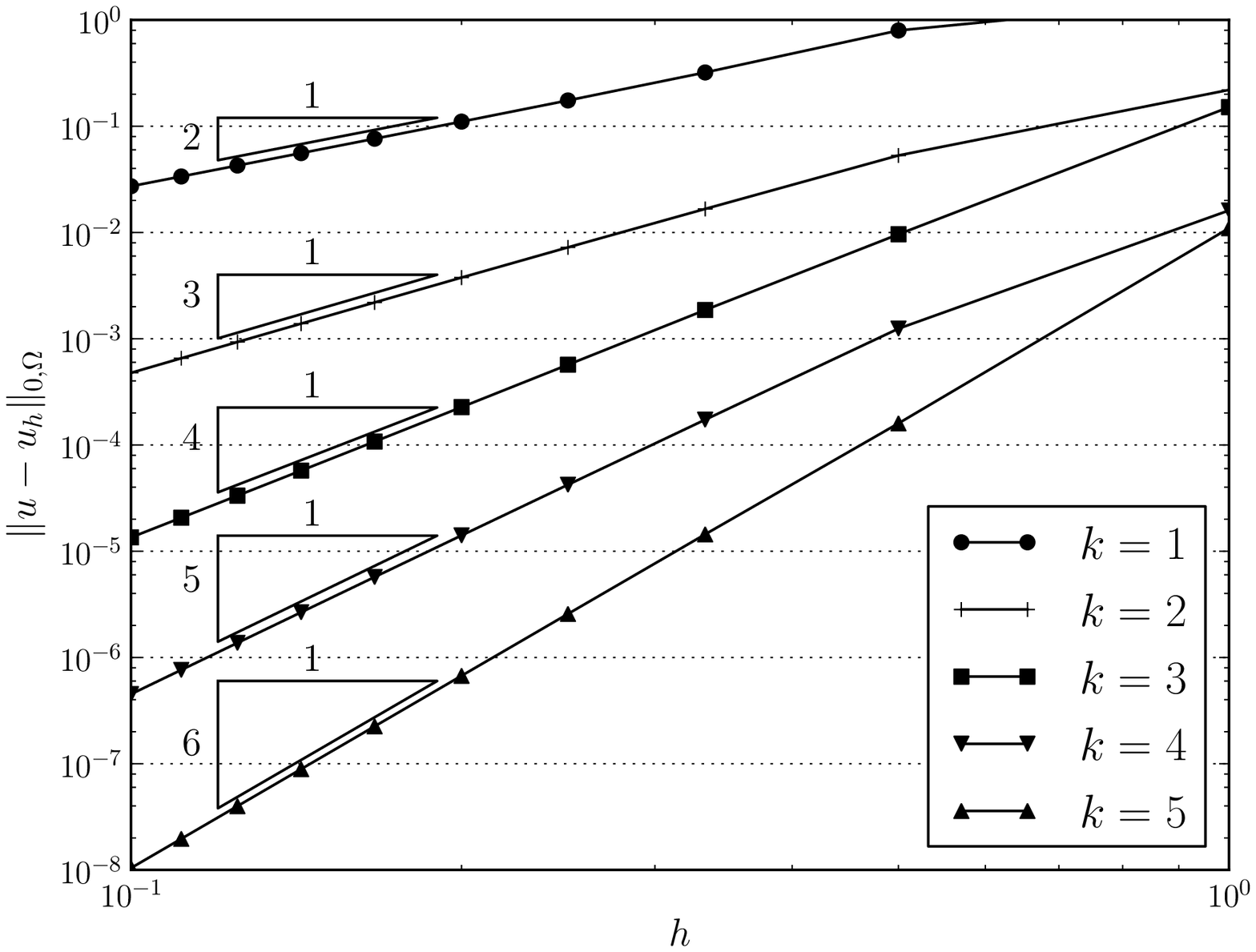}
\caption{Convergence for the elliptic case  in $L^{2}$ with $h$-refinement
for various polynomial orders and~$\alpha = 4k^{2}$.}
\label{fig:elliptic_h_4kk}
\end{figure}
\subsection{Advection-diffusion problems}
An advection-diffusion problem is considered on the domain
$\Omega = \brac{-1, 1}^{2}$,
with $\mu = 0$, $\vect{a} = \brac{e^{x}(y\cos y + \sin y), e^{x}y\sin y}$ and
for various values of~$\kappa$. The source term $f$ is chosen such that
equation~\eqref{eqn:elliptic_exact} is the analytical solution.
For all cases, $\alpha = 4k^{2}$.

The convergence behaviour is examined in terms of
$\unorm{\vect{u} - \vect{u}_{h}}_{A} + \unorm{\vect{u} - \vect{u}_{h}}_{D}$.
The computed error for the case $\kappa = 1 \times 10^{-3}$ is presented in
Figure~\ref{fig:advection_diffusion_1e-3_unorm}.
\begin{figure}
  \center\includegraphics[width=0.75\textwidth]{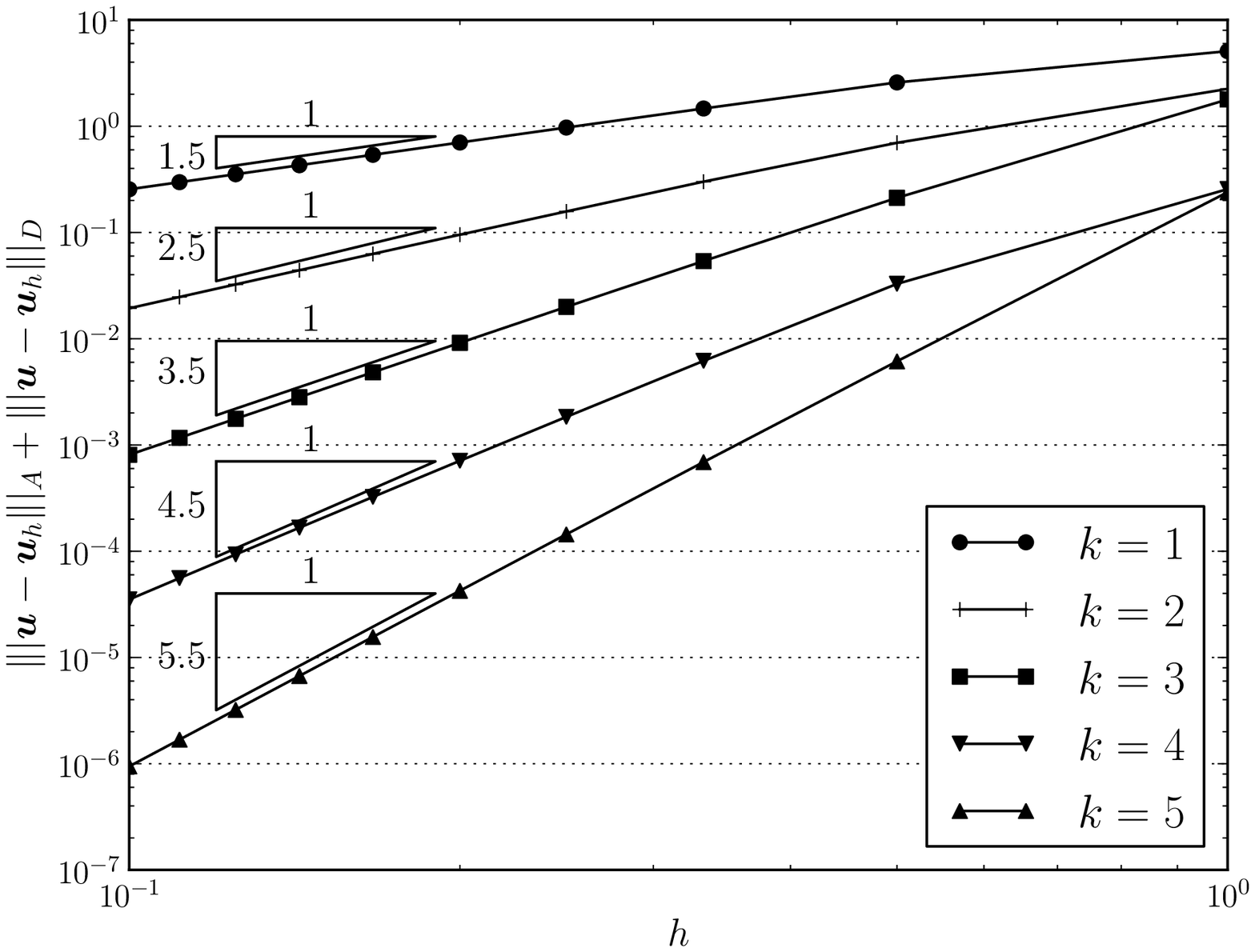}
\caption{Convergence for the advection-diffusion problem in
$\unorm{\vect{u} - \vect{u}_{h}}_{A} + \unorm{\vect{u} - \vect{u}_{h}}_{D}$
for $\kappa = 1\times 10^{-3}$
  with $h$-refinement for various combinations of $k$ and $\alpha = 4k^{2}$.}
\label{fig:advection_diffusion_1e-3_unorm}
\end{figure}
For this advection dominated problem, the method is observed to
converge at the rate $k+1/2$.
For $\kappa =0.1$, the observed convergence response is presented in
Figure~\ref{fig:advection_diffusion_1e-1}.
\begin{figure}
  \center\includegraphics[width=0.75\textwidth]{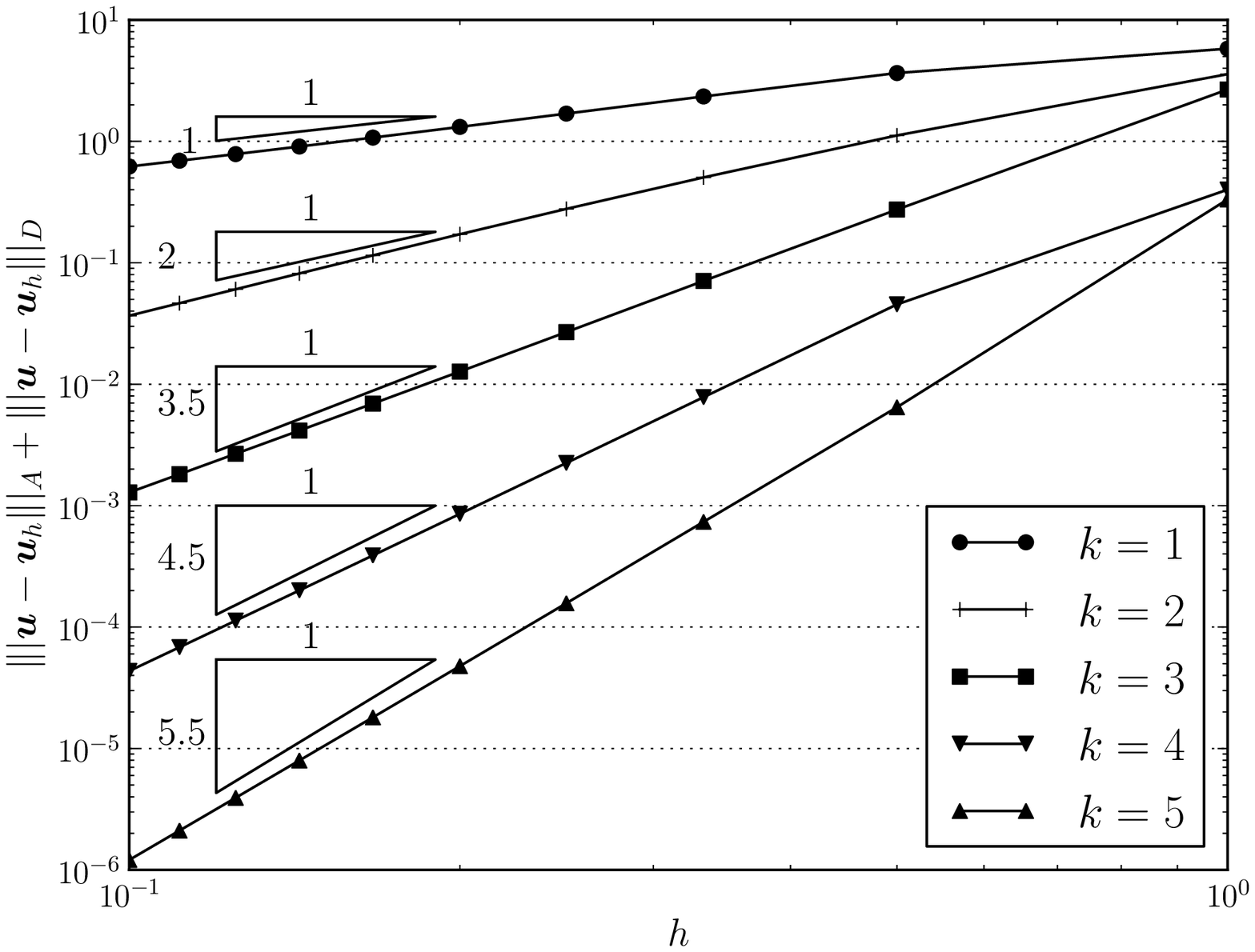}
\caption{Convergence for the advection-diffusion problem in $\unorm{\vect{u} - \vect{u}_{h}}_{A} + \unorm{\vect{u} - \vect{u}_{h}}_{D}$
   for $\kappa = 0.1$
  with $h$-refinement for various combinations of $k$ and $\alpha = 4k^{2}$.}
\label{fig:advection_diffusion_1e-1}
\end{figure}
A convergence rate of $k$ is observed for the lower order polynomial cases,
and the rate appears approach $k + 1/2$ for the higher-order polynomial cases.
For $\kappa = 10$, which is diffusion dominated, the observed convergence
is presented in Figure~\ref{fig:advection_diffusion_10}.
\begin{figure}
  \center\includegraphics[width=0.75\textwidth]{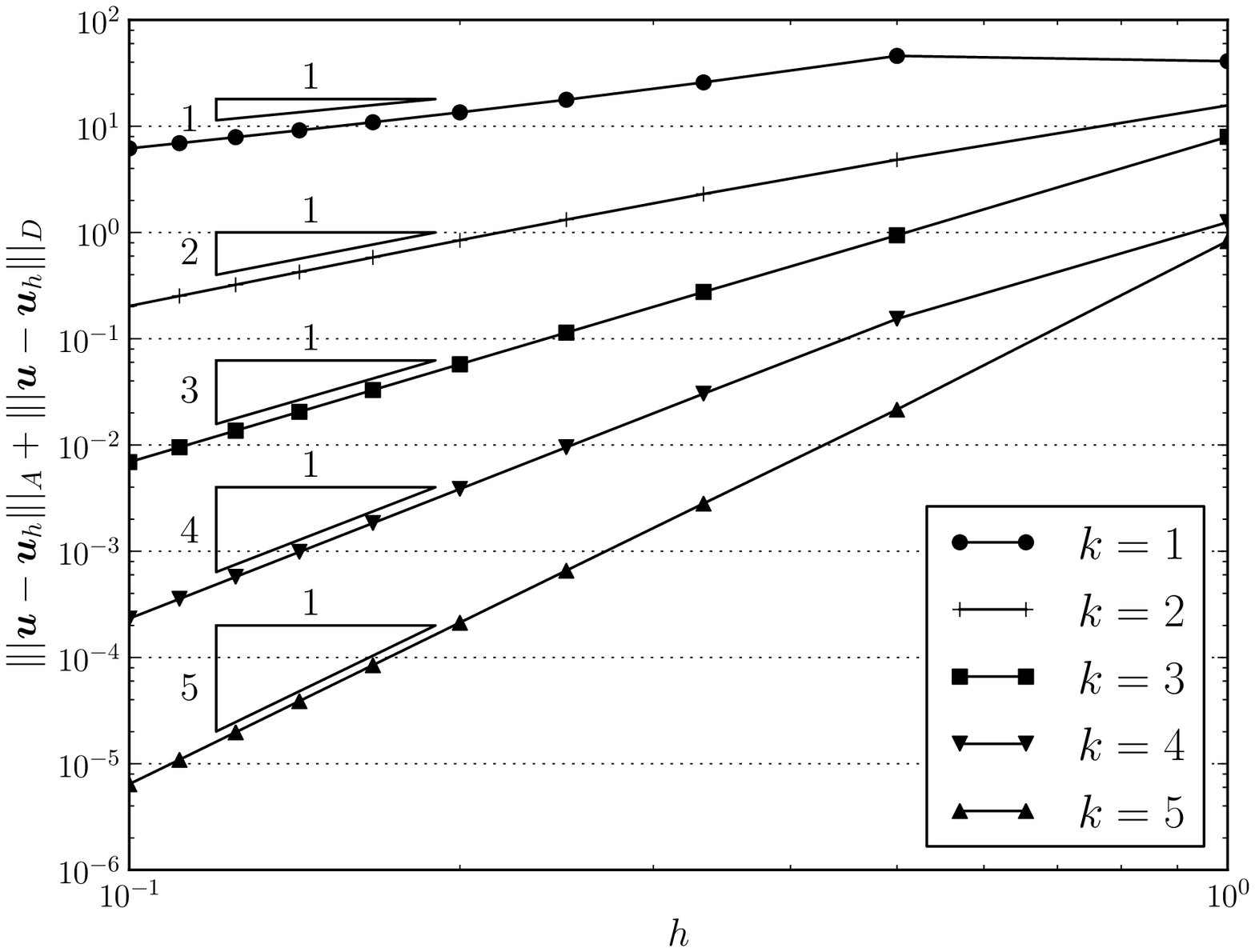}
\caption{Convergence for the advection-diffusion problem  in $\unorm{\vect{u} - \vect{u}_{h}}_{A} + \unorm{\vect{u} - \vect{u}_{h}}_{D}$
for $\kappa = 10$
  with $h$-refinement for various combinations of $k$ and $\alpha = 4k^{2}$.}
\label{fig:advection_diffusion_10}
\end{figure}
As expected, a convergence rate of $k$ is observed for the diffusion-dominated
case.

\section{Conclusions}
Stability and error estimates have been developed for an interface stabilised
finite element method that inherits features of both continuous and
discontinuous Galerkin methods. The analysis is for the hyperbolic and
elliptic limit cases of the
advection-diffusion-reaction
equation.
While the number of global degrees of freedom on a given mesh for the method
is the same as for a continuous finite element
method, the stabilisation mechanism is the same as that present in upwinded
discontinuous Galerkin methods. This is borne out in the stability analysis,
which demands consideration of an inf-sup
condition. Analysis of the method shows that it inherits the stability
properties of discontinuous Galerkin methods, and that it converges in $L^{2}$
at a rate of $k+1/2$ in the advective limit and $k+1$ in the diffusive limit,
as is typical for discontinuous Galerkin and appropriately constructed
stabilised finite element methods.
The analysis presented in this work provides a firm theoretical basis for the method
to support the performance observed in simulations in other works.
The analysis results are supported by
numerical examples which considered a range of polynomial order
elements.
\subsection*{Acknowledgement}
The author acknowledges the helpful comments on this manuscript from
Robert Jan Labeur
and the assistance of Kristian B. {\O}lgaard in
implementing the features that facilitated the automated generation
of the computer code used in Section~\ref{sec:examples}.
\bibliography{references}
\bibliographystyle{unsrtnat}

\end{document}